\documentclass[a4paper]{amsart}
\usepackage{amsmath}
\usepackage{amsfonts}
\usepackage{amssymb}
\usepackage{graphicx}
\usepackage{amsthm}
\usepackage{enumitem}
\usepackage{amsrefs}
\usepackage{hyperref}

\theoremstyle{plain}
\newtheorem{theorem}{Theorem}[section]
\newtheorem{lemma}[theorem]{Lemma}
\newtheorem{prop}[theorem]{Proposition}
\newtheorem{cor}[theorem]{Corollary}
\theoremstyle{definition}
\newtheorem{definition}[theorem]{Definition}
\newtheorem{conv}[theorem]{Convention and notation}
\newtheorem{notation}[theorem]{Notation}
\newtheorem{remark}[theorem]{Remark}

\newcommand{\R}{\mathbb{R}}
\newcommand{\C}{\mathbb{C}}
\newcommand{\F}{\mathbb{F}}

\newcommand{\Kw}{\mathbb{H}}
\newcommand{\G}{\mathbb{G}_r(\mathbb{F}^n)}
\newcommand{\Gg}{\mathbb{G}(\mathbb{F}^n)}

\newcommand{\sS}{\ensuremath{\mathcal{S}}}
\newcommand{\sF}{\ensuremath{\mathcal{F}}}
\newcommand{\sZ}{\ensuremath{\mathcal{Z}}}

\newcommand{\reg}{\ensuremath{\mathcal{R}}}
\newcommand{\GL}{\operatorname{GL}}
\newcommand{\Mat}{\operatorname{Mat}}
\newcommand{\Hom}{\operatorname{Hom}}
\newcommand{\im}{\operatorname{im}}
\newcommand{\func}[3]{#1\colon#2\rightarrow #3}

\newcommand{\tbundle}{\varepsilon_X^n(\F)}
\newcommand{\sC}{\mathcal{C}}
\newcommand{\smooth}{\mathcal{C}^{\infty}}
\newcommand{\VB}{\mathrm{VB}}

\setlist[enumerate,1]{label={\upshape(\arabic*)}} 

\begin{document}
\title{Regulous vector bundles}
\author{Wojciech Kucharz}
\address{Institute of Mathematics, Faculty of Mathematics and Computer Science, Jagiellonian University, ul. {\L{}}ojasiewcza 6, 30-348, Krak{\'o}w, Poland} 
\email{Wojciech.Kucharz@im.uj.edu.pl}
\author{Maciej Zieli{\'n}ski}
\email{Maciej.Zielinski@im.uj.edu.pl}
\begin{abstract}
Among recently introduced new notions in real algebraic geometry is that of regulous functions. Such functions form a foundation for the development of regulous geometry. Several interesting results on regulous varieties and regulous sheaves are already available. In this paper, we define and investigate regulous vector bundles. We establish algebraic and geometric properties of such vector bundles, and identify them with stratified-algebraic vector bundles. Furthermore, using new results on curve-rational functions, we characterize regulous vector bundles among families of vector spaces parametrized by an affine regulous variety. We also study relationships between regulous and topological vector bundles. 
\end{abstract}

\maketitle

\section*{Introduction}
In \cite{KK}, the first named author and K.Kurdyka introduced and investigated a new class of vector bundles on real algebraic varieties, called \textit{stratified-algebraic vector bundles}. Such vector bundles occupy an intermediate position between algebraic vector bundles and topological vector bundles. They have many desirable features of algebraic vector bundles but are more flexible. The main focus of \cite{KK} and the sequel papers \cites{KUCH5, KUCH7, KK2} was on comparison of algebraic, stratified-algebraic, and topological vector bundles. Stratified-algebraic  vector bundles also fit into \textit{regulous geometry} developed by Fichou, Huisman, Mangolte and Monnier \cite{FR}. However, this connection was only briefly indicated in \cite{KK}. Our goal in this paper, which can be viewed as a prequel to \cites{KUCH5, KUCH7, KK, KK2}, is to fill in this gap. 

Along the way, we present a slightly different approach to regulous geometry, based on stratifications and the paper of Koll{\'a}r and Nowak \cite{KN}. Assuming \cite{KN}, we give in Sections 1, 2 and 3 a self-contained exposition of basic notions of regulous geometry. \textit{Regulous vector bundles} are defined and studied in Section 4. The main result is Theorem \ref{bundles_main}. We prove, in particular, that any regulous vector bundle on an affine regulous variety is isomorphic to a subbundle of some trivial vector bundle. This allows us to identify regulous vector bundles with stratified-algebraic vector bundles, cf. Remark \ref{stratified=regulous}. We consider $\F$-vector bundles, where $\F$ denotes $\R$, $\C$, or $\Kw$ (the quaternions). This is justified since $\F$-vector bundles have already proved to be very useful in real algebraic geometry, cf. \cites{BT, BKV, BBK, BCR, BK, BK2, KUCH3, KUCH4, KUCH5, KUCH7, KK, KK2, Zie}. In Section 5, we investigate which families of $\F$-vector spaces on an affine regulous (resp. affine real algebraic) variety are actually regulous (resp. algebraic) $\F$-vector bundles. Our results, Theorems \ref{on_curves} and \ref{on_surfaces}, depend on the recent paper \cite{KKK}. It is a challenging problem to find a characterization of these topological $\F$-vector bundles which admit a regulous structure. We address this problem in Section 6, where the main results are Theorems \ref{6m1} and \ref{real_bundles}.

Regulous geometry is a new subfield of real algebraic geometry. It is fair to say that to regulous geometry belong all the results in which regulous functions, stratified-regular functions or continuous rational functions (these classes of functions are all identical under certain assumptions) play an essential role. For a more complete picture of regulous geometry, the reader may consult \cites{BKV, FR, FMQ, KN, KKK, KUCH, KUCH2, KUCH3, KUCH4, KUCH5, KUCH6, KUCH7, KK, KK2, KK3, KK4, MON, N, Zie}.

Throughout this paper, given real-valued functions $f_1,\ldots, f_k$ on some set $\Omega$, we put $$\sZ(f_1, \ldots, f_k)=\{x\in\Omega :f_1(x)=0, \dots, f_k(x)=0\}.$$

\section{Constructible sets and regulous maps}

In order to fix terminology and notations, we recall some notions introduced in \cite{BCR}, which we recommend for foundations of real algebraic geometry.

A locally ringed space $(V, \reg_V)$, where $\reg_V$ is a sheaf of a real-valued functions on $V$, is called an \textit{affine real algebraic variety} if, for some positive integer $n$, it is isomorphic to an algebraic subset of $\R^n$ endowed with the Zariski topology and the sheaf of real-valued regular functions. More generally, $(V, \reg_V)$ is called a  \textit{real algebraic variety} if there exists a finite open cover $\{U_i\}_{i\in I}$ of $V$ such that each $(U_i, \reg_V\vert_{U_i})$ is an affine real algebraic variety. As usual, we often write $V$ instead of $(V, \reg_V)$ and call $\reg_V$ the \textit{structure sheaf} of $V$. Morphisms of real algebraic varieties are called \textit{regular maps}. The class of real algebraic varieties is identical with the class of quasiprojective real algebraic varieties, cf \cite{BCR}*{Proposition 3.2.10, Theorem 3.4.4}. For this reason, in real algebraic geometry one considers almost exclusively affine varieties. However, non-affine varieties appear in a natural way as the total spaces of pre-algebraic vector bundles, cf. \cite{BCR}*{Chapter 12} and Section 4. Each real algebraic variety carries also the Euclidean topology, which is induced by the standard metric on $\R$.

Recall that a subset $A$ of an affine real algebraic variety $V$ is said to be \textit{constructible} if it belongs to the Boolean algebra generated by the Zariski closed subsets of $V$. Thus $A$ is constructible if and only if it is a finite union of Zariski locally closed subsets of $V$. By a \textit{stratification} of $A$ we mean a finite collection $\sS$ of pairwise disjoint Zariski locally closed subsets of $V$ whose union is $A$.

\begin{prop}\label{stratifications}
Let $V$ be an affine real algebraic variety, $A\subset V$ a constructible subset, and $\sF$ a finite collection of constructible subsets of $V$ with $A\subset\bigcup\sF$. Then there exists a stratification $\sS$ of $A$ such that each stratum of $\sS$ is contained in some element of \sF. Furthermore, such a stratification $\sS$ can be chosen so that each stratum of $\sS$ is an irreducible Zariski locally closed subset of $V$.
\end{prop}

\begin{proof}
All the assertions readily follow  from the definition of constructible sets.
\end{proof}

\begin{definition}\label{function}
Let $V$, $W$ be affine real algebraic varieties, $A\subset V$ a constructible subset, and $\sS$ a stratification of $A$. A map $\func{f}{A}{W}$ is said to be \sS-\textit{regular} if it is continuous   (in the Euclidean topology) and its restriction $f\vert_S$ is a regular map for every stratum $S\in\sS$. Furthermore, $f$ is said to be \textit{regulous} if it is $\mathcal{T}$-regular for some stratification $\mathcal{T}$ of A.
\end{definition}

Regulous maps coincide with stratified-regular maps introduced in \cite{KK} (where only the case $A=V$ is considered).

\begin{prop}\label{basic_functions}
Let $V$, $W$ be affine real algebraic varieties, $A\subset V$ a constructible subset, and $\func{f}{A}{W}$ a regulous map.
\begin{enumerate}
\item If $A_0\subset V$ is a constructible subset contained in $A$, then $f\vert_{A_0}$ is a regulous map. \label{l1:rest}
\item If $B\subset W$ is a constructible subset, then $f^{-1}(B)\subset V$ is a constructible subset. \label{l1:im}
\end{enumerate}
\end{prop}
\begin{proof}
Condition \ref{l1:rest} follows from Proposition \ref{stratifications}.

To prove \ref{l1:im}, let $\sS$ be a stratification of $A$ such that $f\vert_S$ is a regular map for every stratum $S\in\sS$. Then $(f\vert_S)^{-1}(B)$ is a constructible subset of $V$. We get \ref{l1:im} since $f^{-1}(B)=\bigcup_{S\in\sS}(f\vert_S)^{-1}(B)$.
\end{proof}

The definition of a regulous map is local in the following sense.

\begin{prop}\label{locality}
Let $V, W$ be affine real algebraic varieties, $A\subset V$ a constructible subset, and $\{U_1, \ldots, U_r\}$ an open cover of $A$ in the Euclidean topology. Assume that each $U_i$ is a constructible subset of $V$. For a map $\func{f}{A}{W}$, the following conditions are equivalent:
\begin{enumerate}[label=\upshape(\alph*)]
\item $f$ is regulous.
\item $f\vert_{U_i}$ is regulous for $i=1,\ldots,r$
\end{enumerate}
\end{prop}
\begin{proof}
It suffices to apply Proposition \ref{stratifications}.
\end{proof}

The composition of regulous maps is a regulous map. More precisely, the following holds:

\begin{prop}\label{composition}
Let $V$, $W$, $Z$ be affine real algebraic varieties, $A\subset V$, $B\subset W$ constructible subsets, and $\func{f}{A}{W}$, $\func{g}{B}{Z}$ regulous maps with $f(A)\subset B$. Then the map $\func{h}{A}{Z}$, defined by $h(x)=g(f(x))$ for all $x\in A$, is a regulous map.
\end{prop}
\begin{proof}
Once again it is enough to apply Proposition \ref{stratifications}.
\end{proof}

Let $V$ be an affine real algebraic variety and $A\subset V$ a constructible subset. It readily follows form Proposition \ref{stratifications} that the set $\reg^0(A)$ of all regulous functions on $A$ (that is, regulous maps from $A$ into $\R$) forms a subring of the ring $\mathcal{C}(A)$ of all continuous (in the Euclidean topology) functions on $A$.

\begin{prop} Let $V$ be an affine real algebraic variety, $A\subset V$ a constructible subset, and $f_1,\ldots, f_r$ regulous functions on $A$. Then $\sZ(f_1,\ldots, f_r)$ is a constructible subset of $V$.
\end{prop}
\begin{proof}
Setting $f=f_1^2+\ldots+f_r^2$, we get $\sZ(f)=\sZ(f_1,\ldots, f_r)$. It suffices to apply Proposition \ref{basic_functions} \ref{l1:im}.
\end{proof}

\section{The constructible topology}

For any real algebraic variety $V$, we denote by $V^{\text{\upshape ns}}$ its set of nonsingular points, that is,
\begin{displaymath}
V^{\text{\upshape ns}}=\{x\in V\text{: the local ring }\reg_{V,x}\text{ is regular}\}
\end{displaymath}
where $\reg_V$ is the structure sheaf of $V$. Recall that $V^{\text{\upshape ns}}$ is a Zariski open dense subset of $V$.

\begin{prop}\label{nonsingular}
Let $V$ be an affine real algebraic variety. Let $A\subset V$ be a constructible subset that is Euclidean closed and let $W$ be the Zariski closure of $A$ in $V$. Then $W^{\text{\upshape ns}}\subset A$.
\end{prop}

\begin{proof}
By Proposition \ref{stratifications}, $A$ can be written as a finite disjoint union 
$$A=\bigcup_{i\in I}(Z_i\setminus Z'_i)$$ 
where $Z'_i\subset Z_i$ are Zariski closed subsets with $Z_i'\neq Z_i$ and $Z_i$ irreducible for all $i\in I$. Let $W_1,\ldots, W_r$ be the irreducible components of $W$. Since $Z_i$ is irreducible, we have $Z_i\subset W_j$ for some $j$. Setting $I_j=\{i\in I\text{ : }Z_i\subset W_j\}$ we get $$W_j=\bigcup_{i\in I_j}Z_i$$ for all $j=1,\ldots, r$. Furthermore, if $$W_j':=\bigcup_{i\in I_j}Z'_i,$$
then $\dim(W'_j)<\dim(W_j)$ and $$(W_1\setminus W_1')\cup\ldots\cup(W_r\setminus W_r')\subset A.$$
In particular $W_j^{\text{\upshape ns}}\setminus W_j'\subset A$ for every $j$. Since $W^{\text{\upshape ns}}_j\setminus W'_j$ is Euclidean dense in $W^{\text{\upshape ns}}_j$, and $A$ is Euclidean closed, we get $W_j^{\text{\upshape ns}}\subset A$. Hence $$W^{\text{\upshape ns}}\subset W^{\text{\upshape ns}}_1\cup\ldots\cup W_r^{\text{\upshape ns}}\subset A,$$ as required.
\end{proof}
 
\begin{prop}\label{noeth}
Let $V$ be an affine real algebraic variety and let 
$$A_1\supset A_2\supset A_3\supset\ldots$$
 be a sequence of constructible subsets of $V$ that are Euclidean closed. Then the sequence stabilizes, that is, there exists an integer $k>0$ such that $A_i=A_k$ for all $i\geq k$.
\end{prop}

\begin{proof}
Let $W_i$ be the Zariski closure of $A_i$ in $V$ . We use induction on $\dim(W_1)$. In the case $\dim(W_1)=0$ the assertion is obvious since the considered sets are finite. 

Suppose that $\dim(W_1)>0$. The sequence $W_1\supset W_2\supset W_3\supset\ldots$ stabilizes and hence there exists an integer $l>0$ such that $W_l=W_i$ for all $i\geq l$. Set $W:=W_l$. By Proposition \ref{nonsingular}, we get a sequence $$A_l\setminus W^{\text{\upshape ns}}\supset A_{l+1}\setminus W^{\text{\upshape ns}}\supset A_{l+2}\setminus W^{\text{\upshape ns}}\supset\ldots$$ of constructible subsets of $V$ which are Euclidean closed. This sequence stabilizes by the induction hypothesis since $A_l\setminus W^{\text{\upshape ns}}\subset W\setminus W^{\text{\upshape ns}}$ and $\dim(W\setminus W^{\text{\upshape ns}})<\dim(W)\leq\dim(W_1)$. It follows that the sequence $A_l\supset A_{l+1}\supset A_{l+2}\supset\ldots$ stabilizes, which completes the proof.
\end{proof}
Different proofs of Proposition \ref{noeth} are given in \cites{FR, Kur, KPA, N, Par}.
\begin{cor}
Let $V$ be an affine real algebraic variety. The collection of all Euclidean closed constructible sets is the family of closed sets for some topology on $V$, called the \textit{constructible topology}. The constructible topology is Noetherian.
\end{cor}

\begin{proof}
This is a straightforward consequence of Proposition \ref{noeth}.
\end{proof}

Any subset $S\subset V$ can be endowed with the topology induced by the constructible topology on $V$; this topology on $S$ is also Noetherian and called the \textit{constructible topology}. For the sake of convenience we will refer, when it does not lead to any ambiguity, to the open (resp. closed) sets in the constructible topology as constructible open (resp. constructible closed) sets.

The following classical result on semialgebraic functions will play a crucial role in this paper (cf. \cite{BCR}*{Proposition 2.6.4} for the proof).

\begin{prop}\label{Lojasiewicz}
Let $\func{f}{A}{\R}$ be a continuous semialgebraic function defined on a locally closed semialgebraic subset $A\subset\R^n$. Let $\func{g}{A\setminus\sZ(f)}{\R}$ be a continuous semialgebraic function. Then for any integer $N$ large enough, the function 
$$\func{h}{A}{\R}\text{ defined by }h=f^Ng\text{ on } A\setminus\sZ(f)\text{ and } h=0\text{ on }\sZ(f)$$ is continuous. \qed
\end{prop} 

It will be convenient to have the following description of closed sets in the constructible topology, cf. \cite{FR}*{Th{\'e}or{\`e}me 6.4}.

\begin{prop}\label{closed=zero}
Let $V$ be an affine real algebraic variety and let $A\subset V$ be a constructible closed set. Then there exists a regulous function $\func{f}{V}{\R}$ with $\sZ(f)=A$.
\end{prop}
\begin{proof}
We may assume that $V\subset\R^n$ is a Zariski closed subset. Hence, in view of Proposition \ref{basic_functions} \ref{l1:rest}, the proof is reduced to the case $V=\R^n$.

Let $W\subset\R^n$ be the Zariski closure of $A$ in $\R^n$. We use induction on $\dim(W)$. The case $\dim(W)=0$ is obvious since $W$ is then a finite set.

Suppose $\dim(W)>0$. By Proposition \ref{nonsingular}, $W^{\text{\upshape ns}}\subset A$ and hence
$$W=A\cup Z\text{, where } Z=W\setminus W^{\text{\upshape ns}}.$$
Let $\phi$ and $\psi$ be regular functions on $\R^n$ with $\sZ(\phi)=W$  and $\sZ(\psi)=Z.$
Consider the functions
\begin{align*}
&\func{\frac{1}{\phi^2}}{\R^n\setminus W=(\R^n\setminus A)\setminus(Z\setminus(A\cap Z))}{\R}\\
&\func{\psi\vert_{\R^n\setminus A}}{\R^n\setminus A}{\R}.
\end{align*}
Since $\sZ(\psi\vert_{\R^n\setminus A})=Z\setminus(A\cap Z)$, it follows from Proposition \ref{Lojasiewicz} that for a sufficiently large integer $N$, the function $\func{\alpha}{\R^n\setminus A}{\R}$ defined by $$\alpha=\frac{\psi^{2N}}{\phi^2}\text{ on }\R^n\setminus W\text{ and } \alpha=0\text{ on } Z\setminus(A\cap Z)$$ is continuous.

We claim that the function $\func{\beta}{\R^n\setminus(A\cap Z)}{\R}$ defined by $$\beta=\frac{\phi^{2}}{\phi^2+\psi^{2N}}\text{ on }\R^n\setminus Z\text{ and } \beta=1\text{ on } Z\setminus(A\cap Z)$$ is continuous. Indeed, since $$\R^n\setminus(A\cap Z)=(\R^n\setminus A)\cup(\R^n\setminus Z)\text{ and }(\R^n\setminus A)\cap(\R^n\setminus Z)=\R^n\setminus W,$$ it suffices to show that $\beta$ is continuous on $\R^n\setminus A$. However, on $\R^n\setminus A$ we have $\beta=\frac{1}{1+\alpha}$ which proves the claim. Therefore $\beta$ is a regulous function with $$\sZ(\beta)=W\setminus Z=A\setminus(A\cap Z).$$
Since $\dim(Z)<\dim(W)$, it follows from the induction hypothesis that there exists a regulous function $\func{\gamma}{\R^n}{\R}$ with $\sZ(\gamma)=A\cap Z$. By Proposition \ref{Lojasiewicz} once again, if $N'$ is a sufficiently large integer, the function $\func{f}{\R^n}{\R}$ defined by $$f=\gamma^{N'}\beta\text{ on }\R^n\setminus(A\cap Z)\text{ and } f=0\text{ on } A\cap Z$$ is continuous and hence regulous. By construction $\sZ(f)=A$ as required.
\end{proof}

Koll\'ar and Nowak \cite{KN} introduced hereditarily rational functions. It follows from \cite{KN}*{top of page 91} that a function defined on an affine real algebraic variety is continuous (in the Euclidean topology) and hereditarily rational if and only if it is regulous. Thus, \cite{KN}*{Proposition 8} implies the following:

\begin{prop}\label{kollar-nowak}
Let $\func{f}{V}{\R}$ be a function defined on a nonsingular affine real algebraic variety $V$. Then the following conditions are equivalent:
\begin{enumerate}[label=\upshape(\alph*)]
\item $f$ is regulous
\item $f$ is continuous (in the Euclidean topology) and there exists a Zariski open dense subset $V^0\subset V$ such that $f\vert_{V^0}$ is a regular function.
\end{enumerate}
\qed
\end{prop}

Furthermore, \cite{KN}*{Theorem 10} can be stated as follows.

\begin{theorem}\label{tietze_regular}
Let $V$ be an affine real algebraic variety, $W\subset V$ a Zariski closed subvariety, and $\func{f}{W}{\R}$ a regulous function. Then there exists a regulous function $\func{F}{V}{\R}$ such that $F\vert_W=f$. \qed
\end{theorem}

It is not hard to generalize Theorem \ref{tietze_regular}.

\begin{prop}\label{Tietze}
Let $V$ be an affine real algebraic variety, $A\subset V$ a constructible closed subset, and $\func{f}{A}{\R}$ a regulous function. Then there exists a regulous function $\func{F}{V}{\R}$ such that $F\vert_A=f$
\end{prop}
\begin{proof}
Let $W\subset V$ be the Zariski closure of $A$. We use induction on $\dim(W)$. The case $\dim(W)=0$ is obvious since then $W$ is a finite set.

Suppose that $\dim(W)>0$. By Proposition \ref{nonsingular}, $W^{\text{\upshape ns}}\subset A$. Hence $$W=A\cup Z,$$ where $Z=W\setminus W^{\text{\upshape ns}}$. Since $\dim(Z)<\dim(W)$ the induction hypothesis implies that we can find a regulous function $\func{G}{V}{\R}$ with $G\vert_{A\cap Z}=f\vert_{A\cap Z}$. The function $\func{H}{W}{\R}$ defined by $$H=f\text{ on }A\text{ and } H=G\text{ on }Z,$$ is continuous, hence regulous. By Theorem \ref{tietze_regular} there exists a regulous function $\func{F}{V}{\R}$ with $F\vert_W=H$. The proof is complete since $F\vert_A=f$.
\end{proof}

\begin{prop}\label{Tietze2}
Let $V$ be an affine real algebraic variety and let $C\subset B\subset A\subset V$ be constructible closed subsets. For a function $\func{f}{B\setminus C}{\R}$, the following conditions are equivalent:
\begin{enumerate}[label=\upshape(\alph*)]
\item $f$ is regulous.
\item There exists a regulous function $\func{F}{A\setminus C}{\R}$ with $F\vert_{B\setminus C}=f$.
\item for every point $x\in B\setminus C$, there exist an open neighbourhood $U_x\subset A$ of $x$ (in the constructible topology) and a regulous function $\func{f_x}{U_x}{\R}$ such that $f_x=f$ on $U_x\cap(B\setminus C)$.
\end{enumerate}
\end{prop}
\begin{proof}
(a)$\implies$(b) By Proposition \ref{closed=zero}, there exists a regulous function $\func{\phi}{V}{\R}$ with $\sZ(\phi)=C$. According to Proposition \ref{Lojasiewicz}, if $N$ is a sufficiently large integer, then the function $\func{g}{B}{\R}$ defined by $$g=\phi^Nf\text{ on }B\setminus C\text{ and } g=0\text{ on }C$$ 
is continuous, hence regulous. In view of Proposition \ref{Tietze}, we can find a regulous function $\func{G}{V}{\R}$ with $G\vert_{B}=g$. Hence condition (b) holds with $F:=\frac{G}{\phi^N}$ on $A\setminus C$.

(c)$\implies$(a) The constructible topology is Noetherian, and hence $B\setminus C$ is quasi-compact. Consequently, $$B\setminus C\subset U_{x_1}\cup\ldots\cup U_{x_{r}}$$ for some points $x_1,\ldots, x_r$ in $B\setminus C$. Condition (a) holds in view of Propositions \ref{basic_functions} \ref{l1:rest} and \ref{locality}.

The proof is complete since the implication (b)$\implies$(c) follows from Proposition \ref{basic_functions}\ref{l1:rest}.
\end{proof}

Fichou, Huisman, Mangolte and Monnier \cite{FR} defined regulous functions in a different way. It follows from Propositions \ref{kollar-nowak} and \ref{Tietze2} that their definition is equivalent to the one used in this paper.

\section{Regulous varieties}

Let $V$ be an affine real algebraic variety, and $A\subset V$ a closed subset in the constructible topology. We regard $A$ as a topological space endowed with the constructible topology. For every open subset $U\subset A$, the ring $\reg^0(U)$ of regulous functions is defined. The assignment $$\reg^0_A:U\mapsto\reg^0_A(U):=\reg^0(U)$$ is a sheaf of rings on $A$, and $(A, \reg^0_A)$ is a locally ringed space.

\begin{definition}
A locally ringed space $(X,\reg^0_X)$, where $\reg^0_X$ is a sheaf of real-valued functions on $X$, is called an \textit{affine regulous variety} if it is isomorphic to $(A, \reg^0_A)$ for some $A$ as above. More generally, $(X, \reg^0_X)$ is called a \textit{regulous variety} if there exists a finite open cover $\{U_1,\ldots, U_r\}$ of $X$ such that each $(U_i, \reg^0_X\vert_{U_i})$ is an affine regulous variety.

A \textit{morphism} of regulous varieties, called a \textit{regulous map}, is a morphism as locally ringed spaces.
\end{definition}
By abuse of notation we often write simply $X$ for the regulous variety $(X,\reg^0_X)$. We call $\reg^0_X$ the \textit{structure sheaf} of $X$ or the \textit{sheaf of regulous functions} on $X$. If $U\subset X$ is an open subset, then the elements of $\reg^0_X(U)$ are called \textit{regulous functions} on $U$. The underlying topology on $X$, called the \textit{constructible topology}, is Noetherian. An open subset $U\subset X$ for which $(U, \reg^0_X\vert_U)$ is an affine regulous variety is called an \textit{affine open subset} of $X$. By definition, a regulous map $\func{f}{Y}{X}$ of regulous varieties is a continuous map such that for any open subset $U\subset X$ and any regulous function $\varphi$ in $\reg^0_X(U)$, the composite $\varphi\circ f\vert_{f^{-1}(U)}$ is a regulous function in $\reg^0_Y(f^{-1}(U))$. Regulous varieties (together with regulous maps) form a category.

\begin{prop}\label{in_coordinates}
Let $V, W$ be affine real algebraic varieties, $A\subset V$, $B\subset W$ constructible closed subsets, and $\func{f}{A}{B}$ some map. Then $f$ is a regulous map from $(A, \reg^0_A)$ into $(B, \reg^0_B)$ if and only if the composite $\func{i\circ f}{A}{W}$ is a regulous map in the sense of Definition \ref{function}, where $i:B\hookrightarrow W$ is the inclusion map.
\end{prop}

\begin{proof}
We can regard $W$ (and hence $B$) as a subset of $\R^n$ and write $f=(f_1,\ldots, f_n)$ with $\func{f_j}{A}{\R}$. If $\func{f}{(A, \reg^0_A)}{(B, \reg^0_B)}$ is a regulous map, then the functions $f_j$ are regulous, since $f_j=(\pi_j\circ i)\circ f$ and the function $(\pi_j\circ i)$ is clearly regulous ($\pi_k$ denotes the projection from $\R^n$ onto the $k$-th coordinate). The converse implication is just an application of Propositions \ref{basic_functions} and \ref{composition}.
\end{proof}

\begin{prop}\label{open=variety}
Let $(X, \reg^0_X)$ be an affine regulous variety and let $U\subset X$ be an open subset. Then the locally ringed space $(U, \reg^0_X\vert_U)$ is an affine regulous variety.
\end{prop}

\begin{proof}
We regard $X$ as a constructible closed subset of an affine real algebraic variety $V$. Making use of Proposition \ref{closed=zero} we get a regulous function $\func{f}{V}{\R}$ with $\sZ(f)=X\setminus U$. Clearly, $g:V\times\R\ni(x, y)\mapsto f(x)y-1\in\R$ is a regulous function, and $Y=\sZ(g)\cap(X\times\R)$ is a constructible closed subset of $V\times\R$ by Proposition \ref{basic_functions}(2). Then, by Proposition \ref{in_coordinates}, the map $$U\ni x\mapsto\left(x,\frac{1}{f(x)}\right)\in Y$$ is a regulous isomorphism, which completes the proof.
\end{proof}

\begin{prop}\label{sheaf_restriction}
Let $V$ be an affine real algebraic variety and let $B\subset A\subset V$ be constructible closed subsets. Then the restriction $\reg^0_A\vert_B$ coincides with $\reg^0_B$.
\end{prop}

\begin{proof}
The assertion follows from Proposition \ref{Tietze2}.
\end{proof}

\begin{cor}\label{closed=variety}
Let $(X, \reg^0_X)$ be an affine regulous variety and let $Z\subset X$ be a locally closed subset. Then the locally ringed space $(Z, \reg^0_X\vert_Z)$ is an affine regulous variety.
\end{cor}

\begin{proof}
If $Z$ is a closed subset, it suffices to apply Proposition \ref{sheaf_restriction}. The general case also requires Proposition \ref{open=variety}.
\end{proof}

\begin{cor}
Let $(X, \reg^0_X)$ be a regulous variety and let $Z\subset X$ be a locally closed subset. Then the locally ringed space $(Z, \reg^0_X\vert_Z)$ is a regulous variety, and the inclusion map $\func{i}{Z}{X}$ is a regulous map.
\end{cor}

\begin{proof}
The first assertion is a consequence of Corollary \ref{closed=variety}, and the second assertion is obvious.
\end{proof}

\begin{definition}
Let $(X, \reg^0_X)$ be a regulous variety. If $Z\subset X$ is a closed (resp. locally closed) subset, then $(Z, \reg^0_X\vert_Z)$ is called a \textit{regulous subvariety} (resp. a \textit{locally closed regulous subvariety}) of $(X, \reg^0_X)$. We will usually say simply that $Z$ is a regulous subvariety (resp. a locally closed regulous subvariety) of X.
\end{definition}

We now describe a variant of the gluing construction.

\begin{prop}\label{gluing}
Let $\{X_i\}_{i\in I}$ be a finite collection of subsets of a set $X$. Assume that each $X_i$ is endowed with a structure of regulous variety $(X_i, \reg^0_{X_i})$ so that the following conditions hold:
\begin{enumerate}
\item $X_i\cap X_j$ is open in $X_i$ for all $i,j\in I$.
\item $\reg^0_{X_i}\vert_{X_i\cap X_j}=\reg^0_{X_j}\vert_{X_i\cap X_j}$ for all $i, j\in I$
\end{enumerate}
Then there exists exactly one structure of regulous variety $(X, \reg^0_X)$ on $X$ such that each $X_i$ is an open subset of $X$ and $\reg^0_X\vert_{X_i}=\reg^0_{X_i}$
\end{prop}

\begin{proof}
The argument is straightforward.
\end{proof}

To define the product of regulous varieties, we proceed as in the case of algebraic varieties.

Let $V, W$ be affine real algebraic varieties, and $A\subset V$, $B\subset W$ constructible closed subsets. Then $A\times B$ is a constructible closed subset of the affine real algebraic variety $V\times W$. We declare $(A\times B, \reg^0_{A\times B})$ to be the product of $(A, \reg^0_A)$ and $(B, \reg^0_B)$. Given affine regulous varieties $(X, \reg^0_X)$, $(Y, \reg^0_Y)$ and regulous isomorphisms $\func{f}{(X, \reg^0_X)}{(A, \reg^0_A)}$, $\func{g}{(Y, \reg^0_Y)}{(B, \reg^0_B)}$, the product $(X\times Y, \reg^0_{X\times Y})$ is, by definition, an affine regulous variety so that the bijection $\func{f\times g}{X\times Y}{A\times B}$ is a regulous isomorphism. This construction does not depend on the choice of isomorphisms $f$ and $g$.

Suppose now that $(X, \reg^0_X)$, $(Y, \reg^0_Y)$ are regulous varieties. Let $\{X_i\}_{i\in I}$ and $\{Y_j\}_{j\in J}$ be finite open affine covers of $X$ and $Y$, respectively. The products $(X_i\times Y_j, \reg^0_{X_i\times Y_j})$ are already well-defined affine regulous varieties. By applying Proposition \ref{gluing} we obtain a regulous variety $(X\times Y, \reg^0_{X\times Y})$. This construction does not depend on the choice of covers $\{X_i\}_{i\in I}$, $\{Y_j\}_{j\in J}$.

\begin{definition}
The regulous variety $(X\times Y, \reg^0_{X\times Y})$ constructed above is called the \textit{product} of $(X, \reg^0_X)$ and $(Y, \reg^0_Y)$.
\end{definition}

One readily checks that the canonical projections $$\func{\pi_X}{X\times Y}{X}\text{, }\func{\pi_Y}{X\times Y}{Y}$$ are regulous maps. Furthermore, for any regulous variety $Z$, a map $\func{f}{Z}{X\times Y}$ is regulous if and only if the maps $\pi_X\circ f$, $\pi_Y\circ f$ are regulous.

Any regulous variety can be also endowed with the Euclidean topology, which is induced by the standard metric on $\R$.

\begin{prop}\label{regulous=continuous} Any regulous map $\func{f}{X}{Y}$ of regulous varieties is continuous in the Euclidean topology. \end{prop}
\begin{proof}
It suffices to prove that for any affine open subsets $X_0\subset X$, $Y_0\subset Y$ with $f(X_0)\subset Y_0$, the restriction $\func{f\vert_{X_0}}{X_0}{Y_0}$ is a continuous map in the Euclidean topology. This, however, follows from Proposition \ref{in_coordinates}.
\end{proof}

\section{Regulous vector bundles}

Let $\F$ stand for $\R$, $\C$, or $\Kw$ (the quaternions). We will consider only left $\F$-vector spaces which plays a role if $\F=\Kw$ since the quaternions are noncommutative. For consistency, we will consider only left modules over an arbitrary ring.

The $\F$-vector space $\F^n$ is endowed with the standard inner product $$\func{\langle -,-\rangle}{\F^n\times\F^n}{\F}$$ given by $$\langle (x_1,\ldots, x_n), (y_1,\ldots, y_n)\rangle=\sum_{i=1}^{n}x_i\bar{y}_i$$ where $\bar{y}_i$ stands for the conjugate of $y_i$ in $\F$.

\subsection{Matrices and Grassmannians}\label{4A}
Let $\Mat_{m,n}(\F)$, or simply $\Mat_n(\F)$ if $m=n$, denote the $\F$-vector space of all $m$-by-$n$ matrices with entries in $\F$. For any matrix $A=[a_{ij}]\in\Mat_{m,n}(\F)$, the corresponding $\F$-linear transformation $\func{L_A}{\F^n}{\F^n}$ is given by $$(x_1,\ldots,x_n)\mapsto (y_1,\ldots, y_m),\text{ where } y_i=\sum_{j=1}^nx_ja_{ij}\text{ for }i=1,\ldots,m.$$
By abuse of notation we will often write $A(v)$ instead of $L_A(v)$, for $v\in\F^n$. If $B=[b_{jk}]\in\Mat_{n,l}(\F)$, then we define the product $AB=[c_{ik]}$ by $$c_{ik}=\sum_{j=1}^nb_{jk}a_{ij}.$$ This convention implies that $L_{AB}=L_A\circ L_B$.

When convenient, $\F$ will be identified with $\R^{d(\F)}$, where ${d(\F)}=\dim_\R\F$. In particular, we can regard any finite-dimensional $\F$-vector space as a real algebraic variety. The general linear group $\GL_n(\F)$ is a Zariski open subset of the affine real algebraic variety $\Mat_n(\F)$. Furthermore, $\GL_n(\F)$ is a real algebraic group. For $\F=\R$ or $\F=\C$, one justifies these assertions concerning $\GL_n(\F)$ by using determinants. For $\F=\Kw$, one can first embed $\Mat_n(\Kw)$ in $\Mat_{2n}(\C)$ or in $\Mat_{4n}(\R)$ and then apply determinants, cf. for example \cite{QM}.

For any integers $r$ and $n$, with $0\leq r\leq n$, we regard the space $(\F^n)^r$ of all $r$-tuples of vectors in $\F^n$ as an affine real algebraic variety. The Stiefel space $\mathbb{V}_r(\F^n)$ of all $r$-frames in $\F^n$ is a Zariski open subset of $(\F^n)^r$. We also regard the Grassmann space $\G$ of $r$-dimensional $\F$-vector subspaces of $\F^n$ as a real algebraic variety. The canonical projection $$\mathbb{V}_r(\F^n)\ni (v_1,\ldots,v_r)\mapsto\sum_{i=1}^r\F v_i\in\G$$ is a regular map, as is immediately seen by using the standard charts on $\G$. Actually, $\G$ is an affine real algebraic variety and can be identified with $$\mathbb{M}_r(\F^n)=\{A\in\Mat_n(\F) : A^2=A=A^*, \operatorname{trace}(A)=r\},$$ where $A^*$ stands for the conjugate transpose of $A$. The identifcation is effected via the one-to-one correspondence  $$\mathbb{M}_r(\F^n)\ni A\mapsto A(\F^n)\in\G.$$ The treatment of Grassmannians along these lines is contained in \cite{BCR}*{pp. 71-73, 352}.

The disjoint union $$\mathbb{G}(\F^n)=\coprod_{r=0}^n \G$$ is an affine real algebraic variety as well.

\subsection{Vector bundles in real algebraic geometry}\label{4B}

We now recall some well-known facts on $\F$-vector bundles in real algebraic geometry. A detailed presentation is given in \cite{BCR}*{Chapter 12, Section 13.3}.

Let $V$ be an affine real algebraic variety. For any Zariski open subset $U\subset V$, denote by $\reg(U, \F)$ the ring of all regular functions from $U$ into $\F$. The assignment $$\reg_V(\F)\colon U\mapsto\reg_V(\F)(U):=\reg(U,\F)$$ is a sheaf of rings of $\F$-valued functions on $V$. In particular, $\reg_V(\R)$ coincides with the structure sheaf $\reg_V$ of $V$.

For any nonnegative integer $n$, let $$\varepsilon^n_V(\F)=(V\times\F^n,p ,V)$$ denote the product $\F$-vector bundle of rank $n$ on $V$, where $V\times\F^n$ is regarded as a real algebraic variety and $\func{p}{V\times\F^n}{V}$ is the canonical projection. More generally, one considers \textit{pre-algebraic} $\F$\textit{-vector bundles} and \textit{algebraic} $\F$\textit{-vector bundles} on $V$. The category of pre-algebraic $\F$-vector bundles on $V$ is equivalent to the category of locally free $\reg_V(\F)$-sheaves of finite rank on V. A pre-algebraic $\F$-vector bundle on $V$ that is isomorphic (in the category of pre-algebraic $\F$-vector bundles on $V$) to a subbundle of $\varepsilon^n_V(\F)$ is called an algebraic $\F$-vector bundle. In \cite{BCR}*{Example 12.1.5}, there is an example of a pre-algebraic $\R$-line bundle on $\R^2$ that is not algebraic. The category of algebraic $\F$-vector bundles on $V$ is equivalent to the category of finitely generated projective $\reg(V, \F)$-modules.

The tautological $\F$-vector bundle $\gamma_r(\F^n)$ on $\G$ and its orthogonal complement $\gamma^\perp_r(\F^n)$ (with respect to the standard inner product on $\F^n$) are algebraic vector subbundles of $\varepsilon^n_{\G}(\F)$. We denote by $\gamma(\F^n)$ (resp. $\gamma^\perp(\F^n)$) the algebraic $\F$-vector subbundle of $\varepsilon^n_{\mathbb{G}(\F^n)}(\F)$ whose restriction to $\G$ is $\gamma_r(\F^n)$ (resp. $\gamma^\perp_r(\F^n)$).

For a pre-algebraic $\F$-vector bundle $\xi$ on $V$, the following conditions are equivalent:
\begin{enumerate}[label={\upshape (\alph*)}]
\item $\xi$ is an algebraic $\F$-vector bundle.
\item For some $n$, there exists a regular map $\func{f}{V}{\Gg}$ such that the pullback $f^{*}\gamma(\F^n)$ is algebraically isomorphic to $\xi$.
\item For some $n$, there exists a surjective algebraic morphism $\varepsilon^n_V(\F)\rightarrow\xi$.
\item $\xi$ is generated by global algebraic sections.
\item There exists a pre-algebraic $\F$-vector bundle $\eta$ on $V$ such that $\xi\oplus\eta$ is algebraically isomorphic to $\varepsilon_V^n(\F)$ for some $n$.
\item The total space of $\xi$ is an affine real algebraic variety.
\end{enumerate}
The equivalence of conditions (a)-(e) is established in \cite{BCR}*{Sections 12.1, 12.6, 13.3}. It is proven in \cites{HUI, HUI2,MR} that these conditions are equivalent to (f).

\subsection{Vector bundles in regulous geometry}\label{4C}

Throughout the rest of this section we adopt the following
\begin{conv} 
Unless explicitly stated otherwise, we will use constructible topology on regulous varieties. For any regulous variety $Y$ and any nonnegative integer $n$, we will regard $Y\times\F^n$ as a regulous variety. Henceforth, $X$ will denote an affine regulous variety.
\end{conv}

We now introduce our main object of investigation.

\begin{definition}
A \textit{regulous $\F$-vector bundle} on $X$ is a triple $\xi=(E, p, X)$, where $E$ is a regulous variety (not necessarily affine), $\func{p}{E}{X}$ is a regulous map, and the fiber $E_x=p^{-1}(x)$ is an $\F$-vector space for every point $x\in X$. Furthermore, the following local triviality condition holds: For each point $x\in X$, there exist an open neighborhood $U\subset X$ of $x$, an integer $n\geq 0$ (depending on $x$), and a regulous isomorphism $\func{\varphi}{p^{-1}(U)}{U\times\F^n}$ such that $\varphi(E_y)=\{y\}\times\F^n$ and the restriction $\func{\varphi_y}{E_y}{\{y\}\times\F^n}$ of $\varphi$ is an $\F$-linear isomorphism for every $y\in U$.
\end{definition}

We call $E$ (resp. $p$) the \textit{total space} (resp. the \textit{projection}) of the vector bundle $\xi$. Sometimes it will be convenient to use the notation $E(\xi)=E$ and $p(\xi)=p$. We call $(U, \varphi)$ a \textit{vector bundle chart} for $\xi$. Since $X$ is a quasi-compact topological space, there exists a finite collection $\{(U_i,\varphi_i)\}_{i\in I}$ of vector bundle charts for $\xi$ such that the $U_i$ cover $X$. If $Z\subset X$ is a locally closed regulous subvariety, then the restriction $\xi\vert_{Z}=(p^{-1}(Z),p\vert_{p^{-1}(Z)},Z)$ is a regulous $\F$-vector bundle on $Z$. The \textit{rank} of $\xi$ is the function
$$\func{\operatorname{rank}\xi}{X}{\mathbb{Z}}\text{ defined by }x\mapsto\dim_\F(E_x).$$
If $(\operatorname{rank} \xi)(x)=r$ for every $x\in X$, then $\xi$ is said to be of constant rank $r$. Since $X$ is a Noetherian topological space, it has finitely many connected components. The restriction of $\xi$ to any connected component of $X$ is a regulous $\F$-vector bundle of constant rank. Hence $\operatorname{rank} \xi$ is bounded by some nonnegative integer $r_0$, that is, $(\operatorname{rank} \xi)(x)\leq r_0$ for all $x\in X$.

\begin{definition}
A \textit{regulous morphism} (or a \textit{morphism} for short) $\func{h}{\xi}{\eta}$ of regulous $\F$-vector bundles $\xi$ and $\eta$ is a regulous map $\func{h}{E(\xi)}{E(\eta)}$ such that $h(E(\xi)_x)\subset E(\eta)_x$ and the restriction $\func{h_x}{E(\xi)_x}{E(\eta)_x}$ of $h$ is an $\F$-linear transformation for every point $x\in X$.
\end{definition}

Regulous $\F$-vector bundles on $X$ (together with regulous morphisms) form a category.

The product $\F$-vector bundle  $\varepsilon^n_X(\F)=(X\times\F^n,p ,X)$ will be regarded as a regulous $\F$-vector bundle. A regulous $\F$-vector bundle $\xi$ on $X$ is said to be \textit{trivial} if it is isomorphic to $\varepsilon^n_X(\F)$ for some $n$. A vector bundle chart $p(\xi)^{-1}(U)\rightarrow U\times\F^n$ for $\xi$, where $U\subset X$ is an open subset, is a regulous isomorphism between the restriction $\xi\vert_U$ and $\varepsilon_U^n(\F)$.

Morphisms of regulous $\F$-vector bundles have the expected local description. Namely, any regulous morphism $\func{h}{\varepsilon_X^n(\F)}{\varepsilon_X^m(\F)}$ is of the form $$h(x,v)=(x,g(x)(v))\text{ for }(x,v)\in X\times\F^n,$$ where $\func{g}{X}{\Mat_{m,n}(\F)}$ is a regulous map.

\begin{prop}\label{bijective}
Let $\func{h}{\xi}{\eta}$ be a morphism of regulous $\F$-vector bundles on $X$. If the map $h$ is bijective then it is an isomorphism.
\end{prop}

\begin{proof}
Our goal is to prove that the inverse map $\func{h^{-1}}{E(\xi)}{E(\eta)}$ is regulous. It suffices to consider $\xi=\eta=\varepsilon_X^n(\F)$. Then $$h(x,v)=(x,g(x)(v))\text{ for }(x,v)\in X\times\F^n,$$ where $\func{g}{X}{\GL_{n}(\F)}$ is a regulous map. Since $$h^{-1}(x,w)=(x,g(x)^{-1}(w))\text{ for }(x,w)\in X\times\F^m,$$ the map $h^{-1}$ is regulous, as required.
\end{proof}

In topology and complex algebraic geometry it is sometimes convenient to represent vector bundles by transition functions. This can be also done in regulous geometry.

\begin{definition}\label{transition_functions}
Let $\xi$ be a constant rank $n$, regulous $\F$-vector bundle on $X$. Let $\{(U_i,\varphi_i)\}_{i\in I}$ be a finite collection of vector bundle charts for $\xi$ such that the sets $U_i$ cover $X$. For $i,j\in I$, we have $$(\varphi_i\circ\varphi_j^{-1})(x,u)=(x, g_{ij}(x)(u))\text{ for }(x,u)\in(U_i\cap U_j)\times\F^n,$$ where $\func{g_{ij}}{U_i\cap U_j}{\GL_n(\F)}$ is a regulous map. The maps $\{g_{ij}\}_{i,j\in I}$ are called the \textit{transition functions for $\xi$ corresponding to}  $\{(U_i,\varphi_i)\}_{i\in I}$. They satisfy the following conditions:
\begin{enumerate}
\item $g_{ij}g_{jk}=g_{ik}$ on $U_i\cap U_j\cap U_k$ for all $i$, $j$, $k$ in $I$ (the cocycle condition),
\item $g_{ij}^{-1}=g_{ji}$ for all $i$, $j$ in $I$,
\item $g_{ii}=I_n$ (the identity matrix).
\end{enumerate}
\end{definition}

\begin{prop}\label{transition}
Let $\{U_i\}_{i\in I}$ be a finite open cover of $X$. Assume that there are given regulous maps  $\func{g_{ij} }{U_i\cap U_j}{\GL_n(\F)}$ satisfying all three conditions of Definition \ref{transition_functions}. Then there exists a constant rank $n$ regulous $\F$-vector bundle $\eta=(E, p, X)$ on $X$ with transition functions $\{g_{ij}\}_{i,j\in I}$ corresponding to some vector bundle charts $\{(U_i,\psi_i)\}_{i\in I}$. Such a vector bundle $\eta$ is uniquely determined up to regulous isomorphism.
\end{prop}

\begin{proof}
We form an equivalence relation on the union $\coprod_{i\in I}\{i\}\times U_i\times\F^n$ by declaring that $(j,x,v)\in\{j\}\times U_j\times\F^n$ and $(k,x,w)\in\{k\}\times U_k\times\F^n$ are equivalent whenever $g_{jk}(x)(w)=v$. Let $E$ denote the quotient space of this equivalence relation and let $\func{p}{E}{X}$ be the map that sends the equivalence class $[(i,x,v)]$ to $x$. The fiber $E_x=p^{-1}(x)$ is in a natural way an $\F$-vector space. The map $$\func{\psi_i}{p^{-1}(U_i)}{U_i\times\F^n}\text{, }[(i,x,v)]\mapsto(x,v)$$ is bijective and its restriction $ E_x\rightarrow\{x\}\times\F^n$  is an $\F$-linear isomorphism for every $x\in U_i$. We transfer the structure of regulous variety on $U_i\times\F^n$ onto $p^{-1}(U_i)$ by using $\psi_i$. Since the collection $\{g_{ij}\}$ satisfies the cocycle condition, Proposition \ref{gluing} makes it possible to endow the triple $\eta=(E, p, X)$ with a structure of regulous $\F$-vector bundle for which the $(U_i, \psi_i)$ are vector bundle charts. Then $\{g_{ij}\}_{i,j\in I}$ is the collection of transition functions for $\eta$ corresponding to $\{(U_i,\psi_i)\}_{i\in I}$.

Let $\eta'$ be another regulous $\F$-vector bundle of rank $n$ on $X$ which admits vector bundle charts $\{(U_i,\psi'_i)\}_{i\in I}$ with transition functions $\{g_{ij}\}_{i,j\in I}$. Then $\func{h}{\eta'}{\eta}$ given by $$h(v)=\psi^{-1}_i(\psi'_i(v))\text{ for } v\in p(\eta')^{-1}(U_i)$$ is a well defined isomorphism of regulous $\F$-vector bundles on $X$.
\end{proof}

Given an open subset $U\subset X$, we denote by $\reg^0(U, \F)$ the ring of all regulous functions from $U$ into $\F$. The assignment $$\reg^0_X(\F)\colon U\mapsto\reg^0_X(\F)(U):=\reg(U,\F)$$ is a sheaf of rings of $\F$-valued functions on $X$. Clearly, $\reg^0_X(\R)$ coincides with the structure sheaf $\reg^0_X$ of $X$.

Proposition \ref{transition} allows us to relate regulous $\F$-vector bundles on $X$ and locally free $\reg^0_X(\F)$-sheaves of finite rank.

Let $\xi$ be a regulous $\F$-vector bundle on $X$. A \textit{regulous section} $\func{s}{U}{\xi}$ of $\xi$ on an open subset $U\subset X$ is a regulous map $\func{s}{U}{E(\xi)}$ such that $s(x)\in E(\xi)_x$ for every point $x\in U$. The set $\Gamma^0(U,\xi)$ of all regulous sections of $\xi$ on $U$ is an $\reg^0(U, \F)$-module. The assignment $$\Gamma^0_X(\xi)\colon U\mapsto\Gamma^0_X(\xi)(U):=\Gamma^0(U,\xi)$$ is an $\reg^0_X(\F)$-sheaf on $X$. If there is a vector bundle chart $p^{-1}(U)\rightarrow U\times\F^r$ for $\xi$, then the restriction $\Gamma^0_X(\xi)\vert_U$ is isomorphic to the $r$-fold direct sum 
$$(\reg^0_X(\F)\vert_U)^{\oplus r}=(\reg^0_X(\F)\vert_U)\oplus\ldots\oplus(\reg^0_X(\F)\vert_U).$$
Consequently, $\Gamma^0_X(\xi)$ is a locally free $\reg^0_X(\F)$-sheaf of finite rank. We can identify $\reg^0_X(\F)$ and $(\reg^0_X)^{\oplus d(\F)}$ as $\reg^0_X$-sheaves. Thus, $\Gamma^0_X(\F)$ is also a locally free $\reg^0_X$-sheaf of finite rank.

To any morphism $\func{h}{\xi}{\eta}$ of regulous $\F$-vector bundles on $X$ we can associate a morphism $\func{\Gamma^0_X(h)}{\Gamma^0_X(\xi)}{\Gamma^0_X(\eta)}$ of $\reg^0_X(\F)$-sheaves on $X$, defined by $$\Gamma^0_X(\xi)(U)=\Gamma^0(U,\xi)\ni s\mapsto h\circ s\in\Gamma^0(U,\eta)=\Gamma^0_X(\eta)(U)$$ for every open subset $U\subset X$.

\begin{theorem}\label{locally_free}
The functor $\xi\mapsto\Gamma^0_X(\xi)$, defined above, is an equivalence of the category of regulous $\F$-vector bundles on $X$ with the category of locally free $\reg^0_X(\F)$-sheaves of finite rank on $X$.
\end{theorem}

\begin{proof}
By Proposition \ref{transition}, the assignment $\xi\mapsto\Gamma^0_X(\xi)$ establishes a one-to-one correspondence between regulous $\F$-vector bundles on $X$ and locally free $\reg^0_X(\F)$-sheaves of finite rank on $X$, provided that objects of either type are considered up to isomorphism. Furthermore, one easily checks that the map of the sets of morphisms $$\operatorname{Mor}(\xi,\eta)\ni h\mapsto\Gamma^0_X(h)\in\operatorname{Mor}(\Gamma^0_X(\xi),\Gamma^0_X(\eta))$$ is bijective. The proof is complete.
\end{proof}

Let $\xi$ and $\xi'$ be regulous $\F$-vector bundles on $X$. As expected, we say that $\xi'$ is a \textit{regulous $\F$-vector subbundle of} $\xi$ if $E(\xi')$ is a closed regulous subvariety of $E(\xi)$, $p(\xi')=p(\xi)\vert_{E(\xi')}$, and the fiber $E(\xi')_x$ is an $\F$-vector subspace of $E(\xi)_x$ for every point $x\in X$.

Given a morphism $\func{h}{\xi}{\eta}$ of regulous $\F$-vector bundles on $X$, we set $$K(h)=\{v\in E(\xi):h(v)=0\}$$ and define two triples $$\ker h=(K(h),p(\xi)\vert_{K(h)},X),$$ $$\im h=(h(E(\xi)),p(\eta)\vert_{h(E(\xi))},X).$$

\begin{prop}\label{kernel}
Assume that the morphism $\func{h}{\xi}{\eta}$ has the following property: For each connected component $Y$ of $X$, there exists an integer $k$ such that the linear transformation $\func{h_x}{E(\xi)_x}{E(\eta)_x}$ has rank $k$ for every point $x\in Y$. Then $\ker h$ is a regulous $\F$-vector subbundle of $\xi$, and $\im h$ is a regulous $\F$-vector subbundle of $\eta$.
\end{prop}

\begin{proof}
Both assertions can be established by slightly modifying the proof of the corresponding result for topological $\F$-vector bundles given in \cite{HUS}*{Chapter 3, Section 8}.
\end{proof}

Proposition \ref{kernel} implies the following.

\begin{cor}\label{surjective}
Let $\func{h}{\xi}{\eta}$ be a surjective morphism of regulous $\F$-vector bundles on $X$. Then $\ker h$ is a regulous $\F$-vector subbundle of $\xi$.
\end{cor}

By combining Propositions \ref{bijective} and \ref{kernel} we get the following.

\begin{cor}
Let $\func{h}{\xi}{\eta}$ be an injective morphism of regulous $\F$-vector bundles on $X$. Then $\im h$ is a regulous $\F$-vector subbundle of $\eta$. Furthermore, $h$ gives rise to an isomorphism $\xi\rightarrow\im h$.
\end{cor}

Let $Y$ be an affine regulous variety, $\theta$ a regulous $\F$-vector bundle on $Y$, and $\func{f}{X}{Y}$ a regulous map. The pullback $f^*\theta$ is a regulous $\F$-vector bundle on $X$. Recall that the total space of $f^*\theta$ is
$$E(f^*\theta)=\{(x,v)\in X\times E(\theta):f(x)=p(\theta)(v)\}$$ and the bundle projection $\func{p(f^ *\theta)}{E(f^*\theta)}{X}$ is given by $(x,v)\mapsto x$. If $\gamma$ is an $\F$-vector subbundle of $\theta$, then $f^*\gamma$ is an $\F$-vector subbundle of $f^*\theta$. Henceforth we identify $f^*\varepsilon^n_Y(\F)$ with $\varepsilon^n_X(\F)$ via the isomorphism given by $$E(f^*\varepsilon^n_Y(\F))=\{(x,(y,v))\in X\times(Y\times\F^n):f(x)=y\}\rightarrow E(\varepsilon^n_X(\F))=X\times\F^n$$ $$(x,(y,v))\mapsto(x,v).$$
Thus, if $\gamma$ is a regulous $\F$-vector subbundle of $\epsilon^n_Y(\F)$, then $f^*\gamma$ is a regulous $\F$-vector subbundle of $\varepsilon^n_X(\F)$.

If an affine real algebraic variety $V$ is regarded as an affine regulous variety, then any pre-algebraic $\F$-vector bundle on $V$ can also be regarded as a regulous $\F$-vector bundle. In particular, the tautological $\F$-vector bundle $\gamma(\F^n)$ on $\Gg$ can be regarded as a regulous $\F$-vector subbundle of $\varepsilon^n_{\Gg}(\F)$.

\begin{prop}\label{pullback}
Let $\xi$ be a regulous $\F$-vector subbundle of $\varepsilon^n_X(\F)$. Then the map $\func{f}{X}{\Gg}$, defined by $\{x\}\times f(x)=E(\xi)_x$ for all $x\in X$, is a regulous map with $f^*\gamma(\F^n)=\xi$
\end{prop}

\begin{proof}
By assumption, the total space $E(\xi)$ is a closed regulous subvariety of $X\times\F^n$, and $\func{p(\xi)}{E(\xi)}{X}$ is the restriction of the canonical projection $X\times\F^n\rightarrow X$. Let $\func{\varphi}{p(\xi)^{-1}(U)}{U\times\F^r}$ be a vector bundle chart for $\xi$. For $i=1,\ldots,r$, we define $\func{\varphi_i}{U}{\F^n}$ by $(x,\phi_i(x)))=\phi^{-1}(x, e_i),$ where $(e_1,\ldots,e_r)$ is the standard basis for $\F^r$. Obviously, $\phi_i$ is a regulous map. Furthermore, for each $x\in U$ the vectors $\varphi_1,\ldots,\varphi_r(x)$ form a basis for the $\F$-vector space $f(x)$. Since the map $$U\ni x\mapsto(\varphi_1(x),\ldots, \varphi_r(x))\in\mathbb{V}_r(\F^n)$$ is regulous, so is the map $\func{f\vert_U}{U}{\G}$ (cf. Subsection \ref{4B}). Consequently, $f$ is a regulous map with $\xi=f^*\gamma(\F^n$.
\end{proof}

Let $\xi$ and $\eta$ be regulous $\F$-vector bundles on $X$. The direct sum $\xi\oplus\eta$ is a regulous $\F$-vector bundle on $X$. The total space of $\xi\oplus\eta$ is $$E(\xi\oplus\eta)=\{(v,w)\in E(\xi)\times E(\eta):p(\xi)(v)=p(\eta)(w)\}$$ and the bundle projection $\func{p(\xi\oplus\eta)}{E(\xi\oplus\eta)}{X}$ is given by $(v,w)\mapsto p(\xi)(v)$.

Suppose now that $\xi$ is a regulous $\F$-vector subbundle of $\varepsilon_X^n(\F)$. Then $\xi=f^*\gamma(\F^n)$, where $\func{f}{X}{\Gg}$ is the regulous map defined in Proposition \ref{pullback}. Furthermore, $\xi^\perp=f^*\gamma^\perp(\F^n)$ is a regulous $\F$-vector subbundle of $\varepsilon^n_X(\F)$. We refer to $\xi^\perp$ as the \textit{orthogonal complement of $\xi$ in $\varepsilon^n_X(\F)$.}

\begin{prop}\label{direct_sum}
Let $\xi$ be a regulous $\F$-vector subbundle of $\varepsilon_n^X(\F)$. Then $$\xi\oplus\xi^\perp\ni(v,w)\mapsto v+w\in\tbundle$$ is an isomorphism of regulous $\F$-vector bundles on $X$.
\end{prop}

\begin{proof}
It is clear that the map under consideration is an injective morphism. Thus it suffices to apply Proposition \ref{bijective}.
\end{proof}

Let $\xi$ be a regulous $\F$-vector bundle on $X$. A regulous section of $\xi$ defined on $X$ is called a \textit{global section}. We say that $\xi$ is \textit{generated by global sections} if there exist global regulous sections $s_1,\ldots, s_n$ of $\xi$ such that the vectors $s_1(x), \ldots, s_n(x)$ span the $\F$-vector space $E(\xi)_x$ for every point $x\in X$.

The following provides a method for constructing global regulous sections.

\begin{prop} \label{Tietze_sections}
Let $\xi=(E, p, X)$ be a regulous $\F$-vector bundle on $X$. Let $\func{f}{X}{\R}$ be a regulous function and $\func{s}{X\setminus \sZ(f)}{E}$ a regulous section of $\xi$. Then for each integer $N$ large enough, the map $\func{u}{X}{E}$ defined by
$$u(x)=f(x)^Ns(x)\text{ for }x\in X\setminus\sZ(f)\text{ and } u(x)=0\text{ for } x\in \sZ(f)$$ is a global regulous section of $\xi$.
\end{prop}

\begin{proof}
Let $\{(U_i, \varphi_i)\}_{i\in I}$ be a finite collection of vector bundle charts for $\xi$ with $\func{\phi_i}{p^{-1}(U_i)}{U_i\times\F^{n_i}}$ such that $\{U_i\}_{i\in I}$ is a cover of $X$. Setting $U=X\setminus\sZ(f)$, we get a regulous map $$\func{\phi_i\circ s\vert_{U\cap U_i}}{U\cap U_i}{U_i\times\F^{n_i}}.$$ Note that $$U\cap U_i=U_i\setminus \sZ(f\vert_{U_i}).$$
By Proposition \ref{Lojasiewicz}, for each integer $N$ large enough, the map $\func{u_i}{U_i}{E}$ defined by $$u_i(x)=f(x)^Ns(x)\text{ for }x\in U_i\setminus \sZ(f\vert_{U_i})\text{ and } u_i(x)=0\text{ for } x\in \sZ(f\vert_{U_i})$$ is a regulous section of $\xi$. Since $u\vert_{U_i}=u_i$ for all $i\in I$, it follows that $u$ is a global regulous section of $\xi$.
\end{proof}

In the following result, we list the main geometric properties of regulous vector bundles.

\begin{theorem}\label{bundles_main}
Let $\xi$ be a regulous $\F$-vector bundle on $X$. Then $\xi$ is generated by $n$ global regulous sections, for some positive integer $n$. Furthermore, the following conditions hold:
	\begin{enumerate}
	\item There exists a surjective regulous morphism $\tbundle\rightarrow\xi$.
	\item $\xi$ is isomorphic to a regulous $\F$-vector subbundle of $\tbundle$.	
	\item There exists a regulous map $\func{f}{X}{\Gg}$ such that $\xi$ is isomorphic to $f^*\gamma(\F^n)$ in the category of regulous $\F$-vector bundles on $X$.	
	\item There exists a regulous $\F$-vector bundle $\eta$ on $X$ such that $\xi\oplus\eta$ and $\tbundle$ are isomorphic in the category of regulous $\F$-vector bundles on $X$.
	\item The total space $E(\xi)$ is an affine regulous variety.
	\end{enumerate}
\end{theorem}

\begin{proof}
	Let $\{(U_i, \varphi_i)\}_{i\in I}$ be a finite collection of vector bundle charts for $\xi$ such that the sets $U_i$ cover $X$. It is clear that the restriction $\xi\vert_{U_i}$ is generated by regulous sections defined on $U_i$. By Proposition \ref{closed=zero}, $U_i=X\setminus\sZ(f_i)$ for some regulous function $\func{f_i}{X}{\R}$. In view of Proposition \ref{Tietze_sections}, $\xi$ is generated by $n$ global regulous sections $s_1,\ldots, s_n$, for some positive integer $n$. The map $$h\colon X\times \F^n\ni(x,(v_1,\ldots, v_n))\mapsto v_1s_1(x)+\ldots+v_ns_n(x)\in E(\xi)$$ is then a surjective regulous morphism $\func{h}{\tbundle}{\xi}$, which proves (1).
	
	By Corollary \ref{surjective}, $\eta=\ker h$ is a regulous $\F$-vector subbundle of $\tbundle$. According to Proposition \ref{direct_sum}, $\eta\oplus\eta^\perp$ is isomorphic to $\tbundle$ and the restriction of $h$ to $\eta^\perp$ is a bijective regulous morphism $\eta^\perp\rightarrow\xi$. Thus, $\eta^\perp$ is isomorphic to $\xi$ by Proposition \ref{bijective}. Consequently, conditions (2) and (4) hold. In view of (2) and Proposition \ref{pullback}, condition (3) holds as well. Condition (5) follows from (2). The proof is complete.

\end{proof}

In Subsection \ref{4B} we recalled the essential distinction between pre-algebraic vector bundles and algebraic vector bundles. Theorem \ref{bundles_main} shows that there is no counterpart of this phenomenon for regulous vector bundles.

The first assertion in Theorem \ref{bundles_main} can be also proved in a different way. Indeed, for any regulous $\F$-vector bundle $\xi$ on $X$, the sheaf $\Gamma^0_X(\xi)$ of its regulous sections is a locally free $\reg^0_X$-sheaf. In particular, $\Gamma^0_X(\xi)$ is a quasicoherent $\reg^0_X$-sheaf. Thus, according to the regulous version of Cartan's theorem A \cite{FR}*{Th{\'e}or{\`e}me 5.46}, the $\reg^0_X$-sheaf $\Gamma^0_X(\xi)$ is generated by global sections. It follows that the $\F$-vector bundle $\xi$ is generated by global regulous sections. However, we think that the simple direct proof given above is of independent interest.

The correspondence between vector bundles and projective modules established by Serre \cite{Serre} and Swan \cite{SWN} also holds in regulous geometry. To describe it we revert to the notation introduced in the paragraph preceding Theorem \ref{locally_free}.

Let $\xi$ be a regulous $\F$-vector bundle on $X$. By Theorem \ref{bundles_main}(4), $\Gamma^0(X,\xi)$ is a finitely generated projective $\reg^0(X,\F)$-module. To any morphism $\func{h}{\xi}{\eta}$ of regulous $\F$-vector bundles on $X$ we can associate a homomorphism $$\Gamma^0(X,\xi)\ni s\mapsto h\circ s\in\Gamma^0(X,\eta)$$ of $\reg^0(X, \F)$-modules.

\begin{theorem}\label{s-s}
The functor $\xi\mapsto\Gamma^0(X, \xi)$ described above is an equivalence of the category of regulous $\F$-vector bundles on $X$ with the category of finitely generated projective $\reg^0(X, \F)$-modules.
\end{theorem}

\begin{proof}
One can easily adapt the proof of the counterpart of this result in the topological framework given in \cite{ATI}*{pp. 30-31}.
\end{proof}

Theorem \ref{bundles_main} allows us to prove the following variant of Proposition \ref{direct_sum}.

\begin{prop}\label{splitting}
Let $\theta$ be a regulous $\F$-vector bundle on $X$, and $\xi$ a regulous $\F$-vector subbundle of $\theta$. Then there exists a regulous $\F$-vector subbundle $\eta$ of $\theta$ such that $$\alpha\colon\xi\oplus\eta\ni(v,w)\mapsto v+w\in\theta$$ is a regulous isomorphism.
\end{prop}

\begin{proof}
In view of Theorem \ref{bundles_main}(2), we may assume that $\theta$ is a regulous $\F$-vector subbundle of $\tbundle$. Hence $\xi$ is also a regulous $\F$-vector subbundle of $\tbundle$. We identify $\xi\oplus\xi^\perp$ with $\tbundle$ via the isomorphism described in Proposition \ref{direct_sum}. Let $\func{h}{\theta}{\xi}$ be the restriction of the canonical projection $\tbundle=\xi\oplus\xi^\perp\rightarrow\xi$. By Proposition \ref{kernel}, $\eta=\ker h$ is a regulous $\F$-vector subbundle of $\theta$. Furthermore, $\alpha$ is a bijective morphism, hence an isomorphism in view of Proposition \ref{bijective}.
\end{proof}

\begin{remark}\label{hom_bundle}
If $\xi$ and $\eta$ are regulous $\F$-vector bundles on $X$, where $\F=\R$ or $\F=\C$, then we have at our disposal the familiar method of constructing new regulous $\F$-vector bundles: $\xi\otimes\eta$, $\Hom(\xi, \eta)$, $\Lambda^k \xi$ (the $k$-th exterior power) and $\xi^\vee$ (the dual bundle). For example, the fiber of $\Hom(\xi, \eta)$ over any point $x\in X$ is the $\F$-vector space of all $\F$-linear transformations from $E(\xi)_x$ into $E(\eta)_x$.

For regulous $\Kw$-vector bundles $\xi$ and $\eta$ on $X$, we can define in the obvious way a regulous $\R$-vector bundle $\Hom(\xi, \eta)$ on $X$, whose fiber over any point $x\in X$ is the $\R$-vector space of all $\Kw$-linear transformations from $E(\xi)_x$ into $E(\eta)_x$.
\end{remark}

We conclude this section by describing how regulous vector bundles can be identified with stratified-algebraic vector bundles introduced in \cite{KK} and further studied in \cites{KUCH5, KUCH7, KK2}

\begin{remark}\label{stratified=regulous}
According to Theorem \ref{bundles_main}(2), the investigation of regulous $\F$-vector bundles on $X$ can be restricted to regulous $\F$-vector subbundles of $\tbundle$, where $n$ is a nonnegative integer. We now give an alternative description of such subbundles. We may assume without loss of generality that $X$ is a constructible closed subset of some affine real algebraic variety $V$.

Suppose that $\xi$ is a regulous $\F$-vector subbundle of $\tbundle$. Let $\func{f}{X}{\Gg}$ be the map defined by
\begin{equation}\label{class}
\{x\}\times f(x)=E(\xi)_x\text{ for all } x\in X.
\end{equation}
By Proposition \ref{pullback}, $f$ is a regulous map and $\xi=f^*\gamma(\F^n)$. In view of Proposition \ref{regulous=continuous}, $f$ is continuous (in the Euclidean topology) and there exists a stratification $\sS$ of $X$ such that for every stratum $S\in\sS$ the restriction $f\vert_S$ is a regular map. Thus, $\xi$ is a topological (in the Euclidean topology) $\F$-vector subbundle of $\tbundle$ with the following property:
\begin{equation}\label{strat}
\text{\parbox{.85\textwidth}{For every stratum $S\in\sS$ the restriction $\xi\vert_S$ is an algebraic $\F$-vector subbundle of $\tbundle$.}}
\end{equation}

Conversely, suppose that $\xi$ a topological $\F$-vector subbundle of $\tbundle$ satisfying condition (\ref{strat}) for some stratification $\sS$ of $X$ (that is, a stratified-algebraic $\F$-vector bundle on $X$). Defining a continuous map $\func{f}{X}{\Gg}$ by (\ref{class}), we get $\xi=f^*\gamma(\F^n)$. In addition, the restriction $f\vert_S$ is a regular map for every stratum $S\in\sS$ (cf. \cite{BCR}*{Sections 12.1, 12.6, 13.3}). It follows that $f$ is a regulous map, and $\xi$ can be regarded as a regulous $\F$-vector subbundle of $\tbundle$.

In conclusion, regulous $\F$-vector bundles can be identified with stratified-algebraic $\F$-vector bundles. Furthermore, one can show in a similar way that morphisms of regulous $\F$-vector bundles can be identified with morphisms of stratified-algebraic $\F$-vector bundles. It should be mentioned that in \cites{KUCH5, KUCH7, KK, KK2} stratified-algebraic $\F$-vector bundles on $X$ are considered explicitly only in the case where $X$ is an affine real algebraic variety.
\end{remark}

\section{Families of vector subspaces}

In this section we present a new approach to regulous and algebraic vector bundles, based on recently obtained results \cite{KKK}. As in Section 4, we let $\F$ denote $\R$, $\C$ or $\Kw$.

\begin{definition}
Let $\Omega$ be a set and let $n$ be a nonnegative integer. A \textit{family of $\F$-vector subspaces of $\F^n$ on $\Omega$} is a triple $\xi=(E,p,\Omega)$, where $E$ is a subset of $\Omega\times\F^n$, $\func{p}{E}{\Omega}$ is the restriction of the canonical projection $\Omega\times\F^n\rightarrow\Omega$, and the fiber $E_x=p^{-1}(x)$ is an $\F$-vector subspace of $\{x\}\times\F^n$ for every point $x\in\Omega$.

The \textit{classifying map} of $\xi$ is the map $\func{f}{\Omega}{\Gg}$ defined by $\{x\}\times f(x)=E_x$ for all $x\in X$
\end{definition}

For any subset $\Sigma\subset\Omega$, the restriction $\xi\vert_\Sigma=(p^{-1}(\Sigma),p\vert_{p^{-1}(\Sigma)},\Sigma)$ is a family of $\F$-vector subspaces of $\F^n$ on $\Sigma$.

\begin{theorem}\label{on_curves}
Let $V$ be an affine real algebraic variety, $X\subset V$ a constructible closed subset, and $\xi$ a family of $\F$-vector subspaces of $\F^n$ on $X$. Then the following conditions are equivalent:
\begin{enumerate}[label={\upshape (\alph*)}]
\item $\xi$ is a regulous $\F$-vector subbundle of $\tbundle$.
\item For each irreducible real algebraic curve $C\subset V$, the restriction $\xi\vert_{X\cap C}$ is a regulous $\F$-vector subbundle of $\varepsilon_{X\cap C}^{n}(\F)$.
\end{enumerate}
\end{theorem}

\begin{proof}
It is clear that (a) implies (b). Suppose now that (b) holds. Setting $\xi=(E, p, X)$, we consider the classifying map $f$ of $\xi$. Then we get $\xi=f^*\gamma(\F^n)$, as families of $\F$-vector subspaces of $\F^n$ on $X$. To prove that (a) holds it suffices to show that $f$ is a regulous map. This can be done as follows. We regard $\Gg$ as a Zariski closed subset of $\R^N$ and write
$$f=(f_1,\ldots,f_N)\colon X\rightarrow\Gg\subset\R^N.$$
By Proposition \ref{pullback}, the restriction $f\vert_{X\cap C}\colon X\cap C\rightarrow\Gg$ is a regulous map, hence the restrictions $\func{f_i\vert_{X\cap C}}{\R}$ are regulous functions. In particular, the functions $f_i\vert_{X\cap C}$ are continuous (in the Euclidean topology). Note that there are only two possibilities: either the set $X\cap C$ is finite or the set $C\setminus (X\cap C)$ is finite. In either case, there exists a Zariski open dense subset $C^0\subset C$ such that $X\cap C^0$ is a Zariski locally closed subset of $C$ and the restrictions $f_i\vert_{X\cap C^0}$ are regular functions. It follows from \cite{KKK}*{Theorem 1.7, Proposition 3.1} that the $f_i$ are regulous functions. Thus, the map $f$ is regulous, as required.
\end{proof}

We also have a result of a similar nature for algebraic $\F$-vector bundles. Recall that for any real algebraic variety $W$, we denote by $W^{\text{\upshape ns}}$ its nonsingular locus.

\begin{theorem}\label{on_surfaces}
Let $V$ be a nonsingular affine real algebraic variety, and $\xi$ a family of $\F$-vector subspaces of $\F^n$ on $V$. Then the following conditions are equivalent.
\begin{enumerate}[label={\upshape (\alph*)}]
\item $\xi$ is an algebraic $\F$-vector subbundle of $\tbundle$.
\item For each irreducible real algebraic surface $S\subset V$, the restriction $\xi\vert_{S^{\text{\upshape ns}}}$ is an algebraic $\F$-vector subbundle of $\varepsilon_{S^{\text{\upshape ns}}}^{n}(\F)$.

\end{enumerate}
\end{theorem}

\begin{proof}
It is clear that (a) implies (b). Suppose now that (b) holds. Arguing as in the proof of Theorem \ref{on_curves}, we need only show that the classifying map 
$$\func{f}{V}{\Gg\subset\R^N}$$
 is regular by considering its component functions $f_i$. By assumption, the restriction $\func{f\vert_{S^{\text{\upshape ns}}}}{S^{\text{\upshape ns}}}{\Gg}$ is a regular map (cf. \cite{BCR}*{Sections 12.1, 12.6, 13.3}), hence the restrictions $f_i\vert_{S^{\text{\upshape ns}}}$ are regular functions. By \cite{KKK}*{Theorem 6.2}, the $f_i$ are regular functions and so is $f$, as required.
\end{proof}

Our final result does not depend on \cite{KKK}, but it fits into the framework of this section.

Let $V$ be a nonsingular affine real algebraic variety. If $\xi$ is a regulous (resp. an algebraic) $\F$-vector subbundle of $\tbundle$, then its total space $E(\xi)$ is a regulous (resp. a nonsingular real algebraic) subvariety of $V\times \F^n$. This observation provides motivation for the following.

\begin{prop}
let $V$ be a nonsingular affine real algebraic variety and let $\xi$ be a regulous $\F$-vector subbundle of $\varepsilon_V^n(\F)$. Assume that the total space $E(\xi)$ is a $\smooth$ submanifold of $V\times\F^n$. Then $\xi$ is an algebraic $\F$-vector subbundle of $\varepsilon_V^n(\F)$.
\end{prop}

\begin{proof}
We claim that $\xi$ is an $\F$-vector subbundle of $\varepsilon_V^n(\F)$ of class $\smooth$. Indeed, since the map $V\rightarrow E(\xi)$, given by $x\mapsto (x,0)$, is of class $\smooth$, the map $\func{p(\xi)}{E(\xi)}{V}$ is a $\smooth$ submersion. Thus, in view of the constant rank theorem, for each point $x\in V$, one can find an open neighborhood (in the Euclidean topology) $U\subset V$ of $x$ and $\smooth$ maps $\func{s_1,\ldots,s_r}{U}{E(\xi)}$ such that $p(\xi)(s_i(y))=y$ and the vectors $s_i(y)$ form a basis of the $\F$-vector space $E(\xi)_y$ for every point $y\in U$. The claim easily follows.

The claim implies that the classifying map $\func{f}{V}{\Gg}$ of $\xi$ is of class $\smooth$. Furthermore, $f$ is a regulous map by Proposition \ref{pullback}, and hence there exists a Zariski open dense subset $X^0\subset X$ such that the restriction $f\vert_{X^0}$ is a regular map. According to \cite{KUCH}*{Proposition 2.1}, $f$ is a regular map. The proof is complete since $\xi=f^*\gamma(\F^n).$
\end{proof}

\section{Regulous versus topological vector bundles}

In this section we generalize to affine regulous varieties some results of \cite{KK} on relationships between stratified-algebraic and topological vector bundles on affine real algebraic varieties. Such generalizations are not quite routine in general. As noted in Remark \ref{stratified=regulous}, stratified-algebraic vector bundles can be identified with regulous vector bundles.

Henceforth, $\F$ will stand for $\R$, $\C$ or $\Kw$. Any regulous $\F$-vector bundle $\xi$ on an affine regulous variety $X$ can be also viewed as a topological $\F$-vector bundle, which is indicated by $\xi^{\text{top}}$. When topological vector bundles are considered, we always regard $X$ as a topological space endowed with the Euclidean topology. We say that a topological $\F$-vector bundle on $X$ \textit{admits a regulous structure} if it is isomorphic to $\xi^{\text{top}}$ for some $\xi$ as above. We denote by $\VB_{\F\text{-reg}}(X)$ and $\VB_{\F}(X)$ the monoids of isomorphism classes (in the appropriate category) of, respectively, regulous and topological $\F$-vector bundles on $X$. Furthermore, we let $K_{\F\text{-reg}}(X)$ and $K_\F(X)$ denote the corresponding Grothendieck groups.

\begin{theorem}\label{forgetful}
Let $X$ be an affine regulous variety which is compact in the Euclidean topology. Then the homomorphisms
$$\VB_{\F\text{\upshape-reg}}(X)\rightarrow\VB_{\F}(X)\text{ , }K_{\F\text{\upshape-reg}}(X)\rightarrow K_{\F}(X)$$
induced by the correspondence $\xi\mapsto\xi^{\text{\upshape top}}$ are injective. Furthermore, a topological $\F$-vector bundle admits a regulous structure if and only if its class in $K_{\F}(X)$ belongs to the image of the second isomorphism
\end{theorem}
\begin{proof}
According to Theorem \ref{s-s}, regulous $\F$-vector bundles on $X$ can be interpreted as finitely generated projective $\reg^0(X,\F)$-modules. In turn, by Swan's theorem \cite{SWN}, topological $\F$-vector bundles on $X$ can be interpreted as finitely generated projective modules over the ring $\sC(X, \F)$ of continuous (in the Euclidean topology) $\F$-valued functions on $X$. Note that $\sC(X, \F)$ is a topological ring with the topology induced by the supremum norm. The subring $\reg^0(X,\F)$ is dense in $\sC(X, \F)$ in view of the Weierstrass approximation theorem. Now, it is sufficient to invoke Swan's theorem \cite{SWN2}*{Theorem 2.2}.
\end{proof}
\begin{remark}
The homomorphisms in Theorem \ref{forgetful} need not be surjective, cf. \cite{KK}*{Example 7.10} and \cite{KUCH5}*{Theorem 1.7}. The reader may consult \cites{KUCH5, KK, KK2} for results and conjectures on the images of these homomorphisms.
\end{remark}
In the rest of this section, it will always be clear from the context when a given regulous vector bundle $\xi$ is viewed as a topological vector bundle, without writing $\xi^\text{top}$.

We will consider continuous maps between regulous varieties, always referring to the Euclidean topology. The following basic approximation result will be useful.

\begin{lemma}\label{approximation}
Let $X$ be an affine regulous variety, $A\subset X$ a closed regulous subvariety, and $\func{f}{X}{\R}$ a continuous function whose restriction $f\vert_A$ is a regulous function. Assume that $X$ is compact in the Euclidean topology. Then for every $\varepsilon>0$ there exists a regulous function $\func{g}{X}{\R}$ such that $g\vert_A=f\vert_A$ and
$$\lvert f(x)-g(x)\rvert<\varepsilon\text{ for all }x\in X.$$
\end{lemma}
\begin{proof}
We may assume that $X\subset \R^n$ is a constructible closed subset. According to Proposition \ref{closed=zero}, $A=\sZ(\alpha)$ for some regulous function $\func{\alpha}{\R^n}{\R}$. The map
$$e:X\ni x\mapsto(\alpha(x),x)\in\R^{n+1}$$
is a regulous embedding. Hence, by replacing $X$ with $e(X)$, we get
$$X\subset\R^{n+1}\subset V:=\mathbb{P}^{n+1}(\R)\text{ and }A=X\cap H,$$
where $\mathbb{P}^{n+1}(\R)$ is real projective $(n+1)$-space and 
$$H=\{(x_0:\ldots:x_n)\in\mathbb{P}^{n+1}(\R):x_0=0\}.$$
By Proposition \ref{Tietze}, there exists a regulous function $\func{\phi}{V}{\R}$ with $\phi\vert_A=f\vert_A$. In view of Tietze's extension theorem, one can find a continuous function $\func{\psi}{V}{\R}$ satisfying $\psi\vert_X=f-\phi\vert_X$ and $\psi\vert_H=0$. By \cite{BK3}*{Lemma 2.1}, there exists a regular function $\func{\gamma}{V}{\R}$ such that $\gamma\vert_H=0$ and
$$\lvert \psi(x)-\gamma(x)\rvert<\varepsilon\text{ for all }x\in V.$$\upshape
The function $g:=(\gamma+\phi)\vert_X$ has all the required properties.
\end{proof}

Given affine regulous varieties $X$ and $Y$, we denote by $\sC(X,Y)$ the space of all continuous maps from $X$ into $Y$, endowed with the compact-open topology. We concentrate our attention on maps with values in $\Gg$ for some fixed $n$ (cf. Subsection 4B for notation). Lemma \ref{approximation} allows us to generalize \cite{KK}*{Theorem 4.10} as follows.

\begin{theorem}\label{6m1}
Let $X$ be an affine regulous variety, $A\subset V$ a closed regulous subvariety, and $\func{f}{X}{\Gg}$ a continuous map whose restriction $f\vert_A$ is a regulous map. Assume that $X$ is compact in the Euclidean topology. Then the following conditions are equivalent:
\begin{enumerate}[label=\upshape(\alph*)]
\item Each neighborhood of $f$ in $\sC(X,\Gg)$ contains a regulous map 
$$\func{g}{X}{\Gg}\text{ with }f\vert_A=g\vert_A.$$
\item The map $f$ is homotopic to a regulous map $\func{h}{X}{\Gg}$ with $h\vert_A=f\vert_A$.
\item The pullback $\F$-vector bundle $f^*\gamma(\F^n)$ on $X$ admits a regulous structure.
\end{enumerate}
\end{theorem}
\begin{proof}
It is well-known that (a) implies (b). If (b) holds, then so does (c) since the $\F$-vector bundles $f^*\gamma(\F^n)$ and $h^*\gamma(\F^n)$ are topologically isomorphic. It remains to prove that (c) implies (a).

Suppose that (c) holds. To allow for a uniform treatment, given $\F$-vector bundles $\theta_1$ and $\theta_2$, we regard $\Hom(\theta_1, \theta_2)$ as an $\R$-vector bundle (cf. Remark \ref{hom_bundle}). We simplify notation by setting $\gamma=\gamma(\F^n)$ and $\varepsilon^k_X=\varepsilon_X^k(\F)$ for any nonnegative integer $k$. In view of Theorem \ref{bundles_main} (2), we can find a positive integer $m$, a regulous $\F$-vector subbundle $\xi$ of $\varepsilon_X^m$, and a topological isomorphism $\func{\varphi}{\xi}{f^*\gamma}$. Regarding $f^*\gamma$ as a topological subbundle of $\varepsilon_X^n$, we get a continuous section $\func{w}{X}{\Hom(\xi,\varepsilon_X^n)}$ defined by
$$w(x)(e)=\varphi(e)\text{ for }x\in X\text{ and }e\in E(\xi)_x.$$
In particular, $\func{w(x)}{E(\xi)_x}{E(\varepsilon_X^n)_x=\{x\}\times\F^n}$ is an injective $\F$-linear transformation with
$$w(x)(E(\xi)_x)=\{x\}\times f(x).$$
The equality $\varepsilon_X^m=\xi\oplus\xi^\perp$ (the isomorphism of Proposition \ref{direct_sum} is used as identification) and the canonical projection $\func{\rho}{\varepsilon_X^m}{\xi}$ allow us to define a continuous section $\func{\bar{w}}{X}{\Hom(\varepsilon_X^m,\varepsilon_X^n)}$ by
$$\bar{w}(x)(e)=w(x)(\rho(e))\text{ for }x\in X\text{ and }e\in E(\varepsilon_X^m)_x=\{x\}\times\F^m.$$
We identify $\Hom(\varepsilon_X^m,\varepsilon_X^n)$ and $\varepsilon_X^{mn}$ as $\R$-vector bundles. Hence, by the Weierstrass approximation theorem, we can find a regulous section $\func{v}{X}{\Hom(\varepsilon_X^m,\varepsilon_X^n)}$, arbitrarily close to $\bar{w}$ in the compact-open topology. If $\func{\sigma}{\varepsilon_X^n=f^*\gamma\oplus(f^*\gamma)^\perp}{f^*\gamma}$ is the canonical projection  and if $\func{i}{f^*\gamma}{\varepsilon_X^n}$ is the inclusion morphism, then
$$\func{\bar{v}=i\circ\sigma\circ v}{X}{\Hom(\varepsilon_X^m,\varepsilon_X^n)}$$
is a continuous section close to $\bar{w}$. Furthermore, the $\F$-linear transformation
$$\func{\bar{v}(x)}{E(\xi)_x}{E(\varepsilon_X^n)_x=\{x\}\times\F^n}$$ is injective and satisfies
$$\bar{v}(x)(E(\xi)_x)=\{x\}\times f(x)\text{ for all }x\in X.$$
Since the map $f\vert_A$ is regulous, the $\F$-vector bundle $(f\vert_A)^*\gamma=(f^*\gamma)\vert_A$ on $A$ is regulous and the restriction $\func{\sigma_A}{\varepsilon_X^n\vert_A}{(f^*\gamma)\vert_A}$ of $\sigma$ is a regulous morphism. Consequently, the restriction $\func{\bar{v}\vert_A}{A}{\Hom(\varepsilon_X^m,\varepsilon_X^n)}$ is a regulous section. By Lemma \ref{approximation}, there exists a regulous section $\func{u}{X}{\Hom(\varepsilon_X^m,\varepsilon_X^n)}$ that is close to $\bar{v}$ and satisfies $u\vert_A=\bar{v}\vert_A$. Furthermore, the $\F$-linear transformation $$\func{u(x)}{E(\xi)_x}{\{x\}\times\F^n}$$ is injective for all $x\in X$. Then, the map $\func{g}{X}{\Gg}$ defined by
$$u(x)(E(\xi)_x)=\{x\}\times g(x)\text{ for }x\in X$$
is regulous (cf. Subsection 4B), close to $f$, and $g\vert_A=f\vert_A$. Condition (a) follows, which completes the proof.
\end{proof}

Any topological $\F$-vector bundle $\xi$ on an affine regulous variety $X$ can be regarded as an $\R$-vector bundle, which is indicated by $\xi_\R$. In particular, $\xi_\R=\xi$ if $\F=\R$. If the $\F$-vector bundle $\xi$ admits a regulous structure, then so does the $\R$-vector bundle $\xi_\R$.

\begin{theorem}\label{real_bundles}
Let $X$ be an affine regulous variety which is compact in the Euclidean topology. A topological $\F$-vector bundle $\xi$ on $X$ admits a regulous structure if and only if the $\R$-vector bundle $\xi_\R$ admits a regulous structure.
\end{theorem}

If $X$ is an affine real algebraic variety, then Theorem \ref{real_bundles} coincides with \cite{KK}*{Theorem 1.7}. The latter result was rather unexpected and played the key role in the subsequent papers \cites{KUCH5, KK2}. The proof of Theorem \ref{real_bundles} cannot be reduced directly to the already known case and requires some additional preparation.

Let $V$ be an affine real algebraic variety and $W\subset V$ a Zariski closed subvariety. A \textit{multiblowup} $\func{\pi}{V'}{V}$ of $V$ is a regular map which is the composition of a finite sequence of blowups with nonsingular centers. In particular, $V'$ is an affine real algebraic variety, and $\pi$ is a proper map in the Euclidean topology. We say that the multiblowup $\pi$ is \textit{over} $W$ if the restriction $\func{\pi_W}{V'\setminus\pi^{-1}(W)}{V\setminus W}$ of $\pi$ is a biregular isomorphism.

Multiblowups allow us to replace in some arguments affine regular varieties with affine real algebraic varieties. We first describe a general setup.

\begin{notation}\label{multiblowup}
Let $V$ be an affine real algebraic variety and let $X\subset V$ be a constructible closed subset that is Zariski dense in $V$. According to Hironaka's resolution of singularities theorem \cite{HIR} (cf. also \cite{KOL}), there exists a multiblowup $\func{\pi}{V'}{V}$ over the singular locus $W=\text{Sing}(V)$ of $V$, with $V'$ nonsingular. Note that $\pi(V')$ is the Euclidean closure of $V^{\text{ns}}=V\setminus\text{Sing}(V)$ in V. Furthermore, $\pi(V')\subset X$ by Proposition \ref{nonsingular}. We let $\func{\bar{\pi}}{V'}{X}$ denote the regulous map defined by $\bar{\pi}(z)=\pi(z)$ for $z\in V'$
\end{notation}

\begin{lemma}\label{6l1}
Preserving Notation \ref{multiblowup}, assume that $V$ is compact in the Euclidean topology. Let $\xi$ be a topological $\F$-vector bundle on $X$. Then the following conditions are equivalent:
\begin{enumerate}[label=\upshape(\alph*)]
\item The $\F$-vector bundle $\xi$ admits a regulous structure.
\item The $\F$-vector bundles $\xi\vert_{X\cap W}$ on $X\cap W$ and $\bar{\pi}^*\xi$ on $V'$ admit regulous structures
\end{enumerate}
\end{lemma}

\begin{proof}
It is clear that (a) implies (b).

Suppose that (b) holds. Since $X$ is compact in the Euclidean topology, we may assume that $\xi$ is of the form $\xi=f^*\gamma(\F^n)$ for some $n>0$, where $\func{f}{X}{\Gg}$ is a continuous map. In particular, $\xi\vert_{V\cap W}=(f\vert_{X\cap W})^*\gamma(\F^n)$. By Theorem \ref{6m1} (with $A=\emptyset$), the restriction $f\vert_{X\cap W}$ is homotopic to a regulous map $\func{f_0}{X\cap W}{\Gg}$. Recall that $(X, X\cap W)$ is a polyhedral pair, cf. \cite{BCR}*{Theorem 9.2.1}. Thus, in view of the homotopy extension theorem, one can find a continuous map $\func{\phi}{X}{\Gg}$ that is homotopic to $f$ and satisfies $\phi\vert_{X\cap W}=f_0$. The $\F$-vector bundles $f^*\gamma(\F^n)$ and $\phi^*\gamma(\F^n)$ are topologically isomorphic, so the proof is reduced to the case in which $f\vert_{X\cap W}$ is a regulous map.

The subset $A=\pi^{-1}(X\cap W)\subset V'$ is constructible closed, and the restriction $(f\circ\bar{\pi})\vert_A$ of $\pi$ is a regulous map. Furthermore,
$$(f\circ\bar{\pi})^*\gamma(\F^n)=\bar{\pi}^*(f^*\gamma(\F^n))=\bar{\pi}^*\xi$$
By Theorem \ref{6m1}, one can find a regulous map $\func{F}{V'}{\Gg}$ that is arbitrarily close to $f\circ\bar{\pi}$ in $\sC(X,\Gg)$ and satisfies $F\vert_A=(f\circ\bar{\pi})\vert_A$. Since $\pi$ is a proper map and the restriction $\func{\pi_W}{V'\setminus\pi^{-1}(W)}{V\setminus W}$ of $\pi$ is a biregular isomorphism, the map $\func{g}{X}{\Gg}$ defined by
$$g(x)=
\begin{cases}
F(\pi_W^{-1}(x))&\text{for }x\in X\setminus(X\cap W)\\
f(x)&\text{for }x\in X\cap W
\end{cases}$$
is a regulous map. Since $F$ is close to $f\circ\bar{\pi}$, the map $g$ is close to $f$. Hence $g$ is homotopic to $f$, and the $\F$-vector bundles $g^*\gamma(\F^n)$ and $f^*\gamma(F^n)=\xi$ are topologically isomorphic. Condition (a) holds since $g^*\gamma(\F^n)$ is actually a regulous $\F$-vector bundle.
\end{proof}

\begin{proof}[Proof of Theorem \ref{real_bundles}]
As it was already noted, Theorem \ref{real_bundles} holds (being equivalent to \cite{KK}*{Theorem 1.7}) if $X$ is an affine real algebraic variety. This fact together with Lemma \ref{6l1} will yield the proof in the general case. Only the 'if' part requires a proof.

Let $\xi$ be a topological $\F$-vector bundle on $X$ such that the $\R$-vector bundle $\xi_\R$ admits a regulous structure. We may assume that $X$ is a constructible closed subset of real projective $n$-space $\mathbb{P}^n(\R)$. Let $V$ be the Zariski closure of $X$ in $\mathbb{P}^n(\R)$ and let $W$, $\pi$, $\bar{\pi}$ be as in Notation \ref{multiblowup}. In particular, $V'$ is an affine real algebraic variety which is compact in the Euclidean topology, and $\bar{\pi}$ is a regulous map. The proof can be completed by induction on $\dim(V)$. The case $\dim(V)=0$ is obvious. Suppose then that $\dim(V)>0$. Since the $\R$-vector bundle $(\xi\vert_{X\cap W})_\R=(\xi_\R)\vert_{X\cap W}$ on $X\cap W$ admits a regulous structure, so does the $\F$-vector bundle $\xi\vert_{X\cap W}$ on $X\cap W$ by the induction hypothesis. Furthermore, the $\R$-vector bundle $(\bar{\pi}^*\xi)_\R=\bar{\pi}^*(\xi_\R)$ on $V'$ admits a regulous structure, hence the $\F$-vector bundle $\bar{\pi}^*\xi$ on $V'$ also admits a regulous structure. Consequently, by Lemma \ref{6l1}, the $\F$-vector bundle $\xi$ admits a regulous structure, as required.
\end{proof}

\begin{remark}
The problem of describing relationships between algebraic and regulous vector bundles on an affine real algebraic variety is also of interest. The following is contained in \cite{KK}*{Example 1.11}.
Let $\mathbb{T}^n=\mathbb{S}^1\times\ldots\times\mathbb{S}^1$ be the $n$-fold product of the unit circle $\mathbb{S}^1$. If $n\geq 4$, then there exists a regulous $\F$-vector bundle on $\mathbb{T}^n$ which is not topologically isomorphic to any algebraic $\F$-vector bundle.

\end{remark}

\begin{subsection}*{Acknowledgements}
Both authors were partially supported by the National Science Centre (Poland) under grant number 2014/15/B/ST1/00046.
\end{subsection}


\begin{bibdiv}
\begin{biblist}
\bib{ATI}{book}{
	author={Atiyah, M.F},	
	title={K-Theory},
    publisher={W.A Benjamin, New York},
    year={1967},
}
\bib{BT}{article}{
	author={Benedetti, R.},	
	author={Tognoli, A.},	
	title={On real algebraic vector bundles},
	journal={Bull. Sci. Math.},
    volume={104},
	number={2},
    year={1980},
    pages={89--112},
}

\bib{BKV}{article}{
	author={Bilski, M.},	
	author={Kucharz, W.},
	author={Valette, A.},		
	author={Valette, G.},
	title={Vector bundles and regulous maps,},
	journal={Math. Z.},
    volume={275},
    year={2013},
    pages={403--418},

}

\bib{BBK}{article}{
	author={Bochnak, J.},	
	author={Buchner, M.},
	author={Kucharz, W.},
	title={Vector bundles over real algebraic varieties},
	journal={K-Theory},
    volume={3},
    year={1989},
    pages={271--298},
    note={Erratum in K-Theory 4 (1990), 113}
}


\bib{BCR}{book}{
	title={Real Algebraic Geometry},
	author={Bochnak, J.},
	author={Coste, M.},
	author={Roy, M.-F.},
	series={Ergeb. der Math. und ihrer Grenzgeb. Folge (3)},
	volume={36},
    publisher={Springer-Verlag, Berlin},
    year={1998},
}

\bib{BK}{article}{
	author={Bochnak, J.},
	author={Kucharz, W.},
	title={Realization of homotopy classes by algebraic mappings},
	journal={J.Reine Angew. Math.},
    volume={377},
    year={1987},
    pages={159--169},
}

\bib{BK2}{article}{
	author={Bochnak, J.},
	author={Kucharz, W.},
	title={Algebraic approximation of mappings into spheres},
	journal={Michigan Math. J.},
    volume={34},
    year={1987},
    pages={119--125},
}

\bib{BK3}{article}{
	author={Bochnak, J.},
	author={Kucharz, W.},
	title={The homotopy groups of some spaces of real algebraic morphisms},
	journal={Bull. London Math. Soc.},
    volume={25},
    year={1993},
    number={4},
    pages={385--392},
}

\bib{FR}{article}{
	author={Fichou, G.},
	author={Huisman, J.},
	author={Mangolte, F.},
	author={Monnier, J.-Ph.},
	title={Fonctions r\'egulues},
	journal={J. Reine Angew. Math.},
    volume={718},
    year={2016},
    pages={103--151},
}
\bib{FMQ}{article}{
	author={Fichou, G.},
	author={Monnier, J.-P.},
	author={Quarez, R.},
	title={Continuous functions in the plane regular after one blowing up},
	journal={Math. Z.},
	volume={285},
	year={2017},
	pages={287--323},
}

\bib{HIR}{article}{
	author={Hironaka, H.},
	title={Resolution of singularities of an algebraic variety over a field of characteristic zero},
	journal={Ann. of Math.},
	volume={79},
	number={2},
	pages={109--326},
	year={1964},
}


\bib{HUI}{article}{
	author={Huisman, J.},
	title={A real algebraic bundle is strongly algebraic whenever its total space is affine},
	journal={Contemp. Math.},
	volume={182},
	publisher={Amer. Math. Soc., Providence, RI},
	pages={117--119},
	year={1995},
}

\bib{HUI2}{article}{
	author={Huisman, J.},
	title={Correction to "A real algebraic bundle is strongly algebraic whenever its total space is affine"},
	journal={Contemp. Math.},
	volume={253},
	publisher={Amer. Math. Soc., Providence, RI},
	pages={179},
	year={2000},
}

\bib{HUS}{book}{
	title={Fibre Bundles},
	author={Husemoller, D.},
	publisher={Springer, New York Heidelberg},
	year={1975},
}


\bib{KOL}{book}{
	  author={Koll{\'a}r, J.},
      title={Lectures on resolution of singularities},
      series={Ann. of Math. Stud.},
      volume={166},
      year={2007},
      publisher={Princeton University Press, Princeton}
}


\bib{KN}{article}{
	  author={Koll{\'a}r, J.},
	  author={Nowak, K.},
      title={Continuous rational functions on real and {$p$}-adic varieties},
      journal={Math. Z.},
      volume={279},
      year={2015},
      number={1-2},
      pages={85--97},
}

\bib{KKK}{article}{
	author={Koll{\'a}r, J.},
	author={Kucharz, W.},
	author={Kurdyka, K.},
	title={Curve-rational functions},
	journal={Math. Ann.},
	doi={10.1007/s00208-016-1513-z},
	year={2017},
}

\bib{KUCH}{article}{
	author={Kucharz, W.},
	title={Rational maps in real algebraic geometry},
	journal={Adv. Geom.},
	volume={9},
	year={2009},
	pages={517--539},
}

\bib{KUCH2}{article}{
	author={Kucharz, W.},
	title={Regular versus continuous rational maps},
	journal={Topology Appl.},
	volume={160},
	year={2013},
	pages={1086--1090},
}


\bib{KUCH3}{article}{
	author={Kucharz, W.},
	title={Approximation by continuous rational maps into spheres},
	journal={J. Eur. Math. Soc.},
	volume={16},
	year={2014},
	pages={1555--1569},
}

\bib{KUCH4}{article}{
	author={Kucharz, W.},
	title={Continuous rational maps into the unit 2-sphere},
	journal={Arch. Math. (Basel)},
	volume={102},
	year={2014},
	pages={257--261},
}

\bib{KUCH5}{article}{
	author={Kucharz, W.},
	title={Some conjectures on stratified-algebraic vector bundles},
	journal={J. Singul.},
	volume={12},
	year={2015},
	pages={92--104},
}

\bib{KUCH6}{article}{
	author={Kucharz, W.},
	title={Continuous rational maps into spheres},
	journal={Math. Z.},
	volume={283},
	year={2016},
	pages={1201--1215},
}

\bib{KUCH7}{article}{
	author={Kucharz, W.},
	title={Stratified-algebraic vector bundles of small rank},
	journal={Arch. Math. (Basel)},
	volume={107},
	year={2016},
	pages={239--249},
}

\bib{KK3}{article}{
	author={Kucharz, W.},
	author={Kurdyka, K.},
    title={Some conjectures on continuous rational maps into spheres},
    journal={Topology Appl.},
    volume={208},
    year={2016},
    pages={17--29},
}

\bib{KK}{article}{
	author={Kucharz, W.},
	author={Kurdyka, K.},
	title={Stratified-algebraic vector bundles},
	journal={J. Reine Angew. Math.},
	year={2016},
	doi={10.1515/crelle-2015-0105},
}

\bib{KK2}{article}{
	author={Kucharz, W.},
	author={Kurdyka, K.},
    title={Comparison of stratified-algebraic and topological K-theory},
    eprint={arXiv:1511.04238 [math.AG]},    
}

\bib{KK4}{article}{
	author={Kucharz, W.},
	author={Kurdyka, K.},
    title={Linear equations on real algebraic surfaces},
    eprint={arXiv:1602.01986 [math.AG]},    
}

\bib{Kur}{article}{
	author={Kurdyka, K.},
	title={Ensembles semi-alg{\'e}briques symm{\'e}triques par arcs},
	journal={Math. Ann},
	volume={282},
	year={1988},
	pages={445--462},
}

\bib{KPA}{article}{
	author={Kurdyka, K.},
	author={Parusiński, A.},
	title={Arc-symmetric sets and arc-analytic mappings},
	book = {
		title = {Arc spaces in real algebraic geometry},
		series = {Panor. Synth{\`e}ses},
		volume = {24},
		publisher = {Soci{\'e}t{\'e} Math{\'e}matique de France},
		address = {Paris},
		year={2007},
	}
	pages={3--67},
}


\bib{MR}{article}{
	author={Marinari, M.G.},
	author={Raimondo, M.},
	title={Fibrati vettoriali su variet{\`a} algebriche definite su corpi non algebricamente chiusi},
	journal={Boll. Un. Mat. Ital. A},
	number={5},
	volume={16},
	year={1979},
	pages={128--136},
}

\bib{MON}{article}{
	author={Monnier, J.-P.},
    title={Semi-algebraic geometry with rational continuous functions},
    eprint={arXiv:1603.04193 [math.AG]},    
}

\bib{N}{article}{
	  author={Nowak, K.J.},
      title={Some results of algebraic geometry over Henselian rank one valued fields},
      journal={Sel. Math. New Ser.},
      year={2017},
      pages={455--495},
      volume={28},
}

\bib{Par}{article}{
	  author={Parusin{\'n}ski, A.},
      title={Topology of injective endomorphisms of real algebraic sets},
      journal={Math. Ann.},
      year={2004},
      volume={328},
      pages={353--372},
}

\bib{Serre}{article}{
	author={Serre, J.-P.},	
	title={Faisceaux alg{\'e}briques coherents},
	journal={Ann. of Math.},
	number={2},
	volume={61},
	year={1955},
	pages={197-278},
}

\bib{SWN}{article}{
	author={Swan, R.G.},	
	title={Vector bundles and projective modules},
	journal={Trans. Amer. Math. Soc.},
	volume={105},
	year={1962},
	pages={201--234},
}

\bib{SWN2}{article}{
	author={Swan, R.G.},	
	title={Topological examples of projective modules},
	journal={Trans. Amer. Math. Soc.},
	volume={230},
	year={1977},
	pages={201--234},
}

\bib{QM}{article}{
	author={Zhang, F.},
    title={Quaternions and matrices of quaternions},
    journal={Linear Algebra Appl.},
    volume={251},
    year={1997},
    pages={21--57},    
}

\bib{Zie}{article}{
	author={Zieli{\'n}ski, M.},
    title={Homotopy properties of some real algebraic maps},
    journal={Homology Homotopy Appl.},
    volume={18},
    year={2016},
    pages={287--294},    
}

\end{biblist}
\end{bibdiv}
\end{document}